\documentclass[12pt, reqno]{amsart}
\usepackage{amsmath, amsthm, amscd, amsfonts, amssymb, graphicx, color}
\usepackage{amssymb,amsmath,amsfonts,latexsym}
\usepackage{bm}
\usepackage[all,cmtip]{xy}
\usepackage{amscd}
\usepackage{fbox}
\DeclareMathOperator*{\Mprod}{\text{\raisebox{0.25ex}{\scalebox{0.8}{$\prod$}}}}

\newtheorem{theorem}{Theorem}[section]

\newtheorem{proposition}[theorem]{Proposition}
\newtheorem{corollary}[theorem]{Corollary}
\theoremstyle{definition}
\newtheorem{definition}[theorem]{Definition}
\newtheorem{example}[theorem]{Example}

\theoremstyle{remark}
\newtheorem{remark}[theorem]{Remark}
\numberwithin{equation}{section}

\usepackage{bm}

\begin{document}

%\today

\setcounter{page}{1}

\title[Dynamics]{{A class of bilateral weighted shift operators, and linear dynamics}}

\author[Das]{Bibhash Kumar Das}
\address{Indian Institute of Technology Bhubaneswar, Jatni Rd, Khordha - 752050, India}
\email{bkd11@iitbbs.ac.in}

\author[Mundayadan]{Aneesh Mundayadan}
\address{Indian Institute of Technology Bhubaneswar, Jatni Rd, Khordha - 752050, India}
\email{aneesh@iitbbs.ac.in}

%\thanks{$^{\textbf{*}}$ corresponding author(E-mail address: \textbf{aneesh@iitbbs.ac.in)} }

%\begin{abstract}
%\end{abstract}

%\vskip 1cm
\begin{abstract}
This article aims to initiate a study of bilateral weighted backward shift operators defined on the spaces $\ell^p_{a,b}(\Omega_{r,R})$ and $c_{0,a,b}(\Omega_{r,R})$ which are Banach spaces of analytic functions on a suitable annulus in the complex plane, having a normalized Schauder basis of the form,
$$
f_n(z):=
   (a_n+b_{n}z)z^{n},\hskip 0.5cm n\in \mathbb{Z}.
   $$
We obtain necessary and sufficient conditions for a weighted shift $B_w$ to be bounded, and find conditions so that $B_w$ is similar to a compact perturbation of a weighted shift on $\ell^p(\mathbb{Z})$. In addition, we study when $B_w$ is hypercyclic, supercyclic, and chaotic. It shown that the zero-one law of orbital limit points does not hold for $B_w$, which is in contrast to the case of weighted shifts on $\ell^p(\mathbb{Z})$. Most of our results are obtained using the matrix form of $B_w$.
\end{abstract}
\maketitle

\noindent
\tableofcontents
\section{Introduction}

We introduce certain analytic function spaces, and study the linear dynamical aspects of weighted shift operators defined on them. Our work is largely motivated by the relevance of weighted shifts in operator theory and linear dynamics. A bilateral weighted shift $B_w$ on the sequence space $\ell^p(\mathbb{Z})$ is the operator given by $e_n\mapsto w_ne_{n-1}$, $n\in \mathbb{Z}$, where $\{e_n\}$ is the standard basis in $\ell^p(\mathbb{Z})$, and $w=\{w_n\}$ is a bounded complex sequence. We will see that the dynamical behavior of weighted shift operators studied in this paper are distinctive in comparison with the above classical weighted shifts.
For general properties of weighted shifts, we refer to Halmos \cite{Halmos} and Shields \cite{Shields}.

We first recall the following notions, and refer to Bayart and Matheron \cite{Bayart-Matheron}, and Grosse-Erdmann and Peris \cite{Erdmann-Peris} for details on the subject of linear dynamics which, by now, has become a classic on its own. 

 \begin{definition}
 Let $X$ be a separable topological vector space over $\mathbb{K}:=\mathbb{C}$ or $\mathbb{R}$. A linear operator $T:X\rightarrow X$ is said to be
 \begin{itemize}\item[--] \textit{supercyclic} if the set 
 $\{\lambda T^nx:\lambda \in \mathbb{K}, n\geq 0\}$
 is dense in $X$ for some vector $x\in X$,
 \item[--] \textit{hypercyclic} if there exists $x\in X$ whose orbit
 \[
 \text{Orb}(T,x):=\{x,Tx,T^2x,\cdots\}
 \]
 is dense in $X$,
\item[--] \textit{topologically transitive} if, for any two non-empty open sets $U_1$ and $U_2$ of $X$, the set $\{k\in \mathbb{N}: T^k(U_1)\cap U_2\neq \varnothing\}$ is non-empty, 
\item[--] \textit{topologically mixing} if, for any two non-empty open sets $U_1$ and $U_2$ of $X$, the set $\{k\in \mathbb{N}: T^k(U_1)\cap U_2\neq \varnothing\}$ is co-finite, and 
\item[--] \textit{chaotic} if it is hypercyclic with a dense subspace of periodic vectors. (Recall that vector $v$ is periodic for $T$ if $T^kv=v$ for some $k$.)
 \end{itemize}
 \end{definition}

Salas provided a complete characterization of hypercyclic weighted shifts \cite{Salas-hc}. See, also, Rolewicz \cite{Rolewicz} for the hypercyclicity of $\lambda B$ on $\ell^p(\mathbb{N}_0)$, $1\leq p<\infty$, and $|\lambda|>1$, Kitai \cite{Kitai} and Gethner-Shapiro \cite{Gethner-Shapiro} for some standard criterion with applications to weighted shifts. The notion of supercyclicity was introduced by Hilden and Wallen \cite{Hilden-Wallen} who established that every weighted backward shift $B_w$ (with non-zero weights) is supercyclic. A characterization of supercyclic weighted shifts was obtained by Salas in \cite{Salas-sc}. Concerning the study of chaos within the frame work of linear operators, Godefroy and Shapiro \cite{Godefroy-Shapiro} provided several classes of chaotic operators including weighted shifts. Grosse-Erdmann \cite{Erdmann} characterized the hypercyclicity and chaos for $B_w$ in a large class of sequence spaces. For details and developments in linear dynamics, we refer to Bayart and Matheron \cite{Bayart-Matheron}, Grosse-Erdmann and Peris \cite{Erdmann-Peris}, and Grivaux et al. \cite{Grivaux}.

In \cite {Kit1}, Chan and Seceleanu introduced and established a zero-one law for hypercyclicity. Their result says that if a weighted shift $B_w$ on $\ell^p(\mathbb{Z})$ has an orbit admitting non-zero limit points, then $B_w$ is hypercyclic, where $1\leq p<\infty$. A similar result holds for unilateral weighted shifts on $\ell^p(\mathbb{N}_{0})$. Bonilla et al. \cite{Bonilla-Cardeccia} generalized this zero-one law to a large class of weighted shifts on unilateral and bilateral sequence spaces, but noted that the law fails in $c_{0}(\mathbb{Z})$. Motivated by the importance of the classical unilateral shifts, the authors in \cite{Das-Mundayadan} and \cite{Das-Mundayadan1} introduced the spaces $\ell^p_{a,b}$ and $c_{0,a,b}$ having normalized Schauder basis of the form $\{(a_n+b_nz)z^n:n\geq 0\}$, and studied the dynamics of weighted shifts on these spaces. In particular, it was observed in \cite{Das-Mundayadan1} that the adjoint of certain shifts have non-trivial periodic vectors, but the adjoint is not even hypercyclic. Moreover, the orbital zero-one law fails for the adjoint. For related work, we refer to Chan and Seceleanu \cite{Kit}, and Abakumov and Abbar \cite{Abakumov-Abbar}.

In this paper, we initiate a study of bilateral analogue. The paper is organized as follows. In Section $2$, the spaces $\ell^p_{a,b}(\Omega_{r,R})$ and $c_{0,a,b}(\Omega_{r,R})$ are introduced, and the continuity of evaluation and coordinate functionals are established. (These spaces are of independent interests as well.) In Section $3,$ we obtain necessary and sufficient conditions for the boundedness of a bilateral weighted backward shift $B_w$. In Section $4$, as main results, we study when $B_w$ is hypercyclic, mixing, supercyclic, and chaotic. Further, the zero-one law of hypercyclicity is proved for $B_w$ under certain assumptions; however, it is proved that the zero-one law of hypercyclicity is not true for $B_w$ on $\ell^p_{a,b}(\Omega_{r,R})$ in general.

\textsf{Throughout the paper, we will always assume the following}:
\begin{itemize}
\item[(1)] 
     $a:=\{a_n\}_{n=-\infty}^{+\infty}, b:=\{b_n\}_{n=-\infty}^{+\infty}~ are~ complex~ sequences,~ where~ a_n\neq 0, n\in \mathbb{Z}.$
     \item[(2)] \[
     \sup_{n\leq -1} \left |\frac{a_{n+1}a_{n+2}\cdots a_0}{b_nb_{n+1}\cdots b_{-1}}\right|=\infty.
     \]
 \item[(3)] $
 f_n(z)=
   (a_n+b_{n}z)z^{n}, \hskip.5cm n\in \mathbb{Z}.$
\item[(4)] \begin{equation}\label{r-R}
 r:\limsup_{n\rightarrow +\infty} ~(|a_{-n}|+|b_{-n}|)^{1/n}<R:=\frac{1}{\limsup_{n\rightarrow +\infty} ~(|a_n|+|b_n|)^{1/n}}\neq 0,
\end{equation}
and 
\item[(5)]
$\Omega_{r,R}:~ \text{the~ annulus}~ r<|z|<R.$

\end{itemize}
The importance of these assumptions will be clear in the next section. Indeed, we require the above conditions for constructing the analytic function spaces over $\Omega_{r,R}$, studied in this paper. 

\section{The spaces $\ell^p_{a,b}(\Omega_{r,R})$ and $c_{0,a,b}(\Omega_{r,R})$}

In this section, we will define the above mentioned spaces of analytic functions on the annulus $\Omega_{r,R}$, having a normalized Schauder basis of the form $\{f_n:-\infty<n<+\infty\}$, as given above. In particular, we will determine a condition so the space contains all Laurent polynomials, and show that the evaluation functionals are continuous. Also, we find norm estimates for coordinate functionals. These results will be essential for the study of the dynamics of weighted shifts. We recall that the unilateral versions $\ell^p_{a,b}$ and $c_{0,a,b}$ were introduced and studied in \cite{Das-Mundayadan} and \cite{Das-Mundayadan1}. 

\begin{proposition} \label{Annuli}
    For $1\leq p\leq \infty,$ and $\{\lambda_n\}\in \ell^p(\mathbb{Z})$, the series 
    \begin{equation}\label{Laurent}
        f(z):=\sum_{n=-\infty}^{+\infty}\lambda_{n}f_{n} (z)
    \end{equation}
    converges uniformly and absolutely on compact subsets of $\Omega_{r,R}$. Moreover, the representation of $f(z)$ in terms of $\{\lambda_n\}\in \ell^p(\mathbb{Z})$ and $\{f_n\}$ is unique.
    
    The same conclusions hold true if we take $\{\lambda_n\}$ from $c_0(\mathbb{Z})$.

\end{proposition}
\begin{proof} 

%Set 
  %  \begin{center}
  %  $\mathtt{A}_{1}:=\left\{ \zeta \in \mathbb{C}:\sum_{n=-\infty}^{-1}(|a_{n}||%\zeta|^{n})^{q}<\infty~~~~~\text{and}~~~~~\sum_{n=0}^{+\infty} \big(( \lvert %a_{n}\rvert+\lvert b_{n}\rvert)\lvert \zeta \rvert^{n} \big)^{q}<\infty\right\}$,
  %  \end{center}
   % where $\frac{1}{p}+\frac{1}{q}=1,$
    % \begin{center}
   % %$\mathtt{A}_{2}:=\left\{ \zeta \in \mathbb{C}:\sup_{n <0} ~\lvert %a_{n}\rvert\lvert \zeta\rvert^{n}<\infty~~~~~\text{and}~~~~~\sup_{n\geq 0} %~(\lvert a_{n}\rvert+\lvert b_{n}\rvert)\lvert \zeta\rvert^{n}<\infty\right\},$
  %  \end{center}
  %  and
   % \begin{center}
   % $\mathtt{A}_{3}:=\left\{ \zeta \in \mathbb{C}:\sum_{n=-\infty}^{-1} \lvert %a_{n}\rvert\lvert \zeta \rvert^{n} <\infty~~~~~\text{and}~~~~~\sum_{n=0}^{+\infty} ( \lvert a_{n}\rvert+\lvert b_{n}\rvert)\lvert \zeta \rvert^{n} <\infty\right\}.$
  %  \end{center}

We prove the result for $1<p<\infty$, and the other cases are similar. Recall that radii $r$ and $R$ of  $\Omega_{r,R}$ are given by 
\[
r=\limsup_{n\rightarrow +\infty} ~(|a_{-n}|+|b_{-n}|)^{1/n}~~~~~~\text{and}~~~~~~\frac{1}{R}=\limsup_{n\rightarrow +\infty} ~(|a_n|+|b_n|)^{1/n}.
\]
Then, for a compact set $K$ in  $\Omega_{r,R}$, $\zeta \in K$, we note that 
    \begin{eqnarray*}
    \sum_{n=-\infty}^{+\infty} \lvert \lambda_{n}f_{n}(\zeta)\rvert&=&\sum_{n=-\infty}^{+\infty} \lvert \lambda_{n}(a_{n}+b_{n}\zeta)\zeta^{n}\rvert\\
    &\leq&\max_{\zeta\in K} \{1, \lvert \zeta\rvert\}\sum_{n=-\infty}^{+\infty}\lvert\lambda_{n}\rvert(\lvert a_{n}\rvert+\lvert b_{n}\rvert)\lvert \zeta \rvert^{n}\\
    &=&\max_{\zeta\in K} \{1, \lvert \zeta\rvert\} \left( \sum_{n=-\infty}^{+\infty}(\lvert a_{n}\rvert  +\lvert b_{n}\rvert)^{q}\lvert \zeta \rvert^{qn}\right)^{\frac{1}{q}} \left(\sum_{n=-\infty}^{+\infty} \lvert\lambda_{n}\rvert^p\right)^{\frac{1}{p}}\\
    &=& M\lVert f \rVert,
     \end{eqnarray*}
      where  $M=\max_{\zeta\in K} \{1, \lvert \zeta\rvert\} \left( \sum_{n=-\infty}^{+\infty}(\lvert a_{n}\rvert+\lvert b_{n}\rvert)^{q}\lvert \zeta \rvert^{qn}\right)^{\frac{1}{q}}$.      Hence, by our hypothesis the series $\sum_{n=-\infty}^{+\infty}\lambda_{n}f_{n} $ converges absolutely and uniformly on compact subsets of  $\Omega_{r,R}.$ 

      \noindent To verify the uniqueness, if $\lambda=\{\lambda_{n}\}_{n\in\mathbb{Z}}\in \ell^p(\mathbb{Z})$, and
\begin{equation}\label{uniqueness}
    0=\sum_{n=-\infty}^{+\infty} \lambda_{n}f_{n}(z)=\sum_{n=-\infty}^{+\infty}\lambda_{n}(a_{n}+b_{n}z)z^{n},
\end{equation}
then a comparison of coefficients implies that $\lambda_n=0$ for all $n$. Indeed, a simple calculation shows that
\[
\lambda_{n}=
\begin{cases}
\displaystyle
(-1)^{n}\lambda_{0}\,\frac{b_{0} b_{1} \cdots b_{n-1}}{a_{1} a_{2} \cdots a_{n}}, & n \geq 1,\\[1.2ex]
\lambda_{0}, & n = 0, \\[1.2ex]
\displaystyle(-1)^{-n}\lambda_{0}\,
\frac{a_{n+1} a_{n+2} \cdots a_{0}}{b_{n} b_{n+1} \cdots b_{-1}},
& n \leq -1 .
\end{cases}
\]
Now, by our assumption $ \sup_{n\leq -1} \left |\frac{a_{n+1}a_{n+2}\cdots a_0}{b_nb_{n+1}\cdots b_{-1}}\right|=\infty,$ it follows that $\lambda=\{\lambda_{n}\}_{n\in\mathbb{Z}}=0$ is the only solution $\lambda\in \ell^{p}(\mathbb{Z})$ of \eqref{uniqueness}. This yields the proposition.
\end{proof}

In view of the above proposition, we can now define the following spaces.
\begin{definition}
For $1\leq p < \infty$, define $\ell^p_{a,b}(\Omega_{r,R})$ to be the space of all analytic functions $f(z)$ on the annulus $\Omega_{r,R}$ such that
$f(z)=\sum_{n=-\infty}^{+\infty}\lambda_{n}f_{n}(z)$
with $\{\lambda_n\}_{n\in\mathbb{Z}}\in \ell^p(\mathbb{Z})$, and has the norm
\begin{equation*}
    \lVert f \rVert_{\ell^p_{a,b}(\Omega_{r,R})}:=\left(\sum_{n=-\infty}^{+\infty}\lvert\lambda_{n}\rvert^{p}\right)^{\frac{1}{p}}, ~1\leq p<\infty.
\end{equation*}
In a similar way, we define the spaces $\ell^{\infty}_{a,b}(\Omega_{r,R})$ and  $c_{0,a,b}(\Omega_{r,R})$. The later space consists of analytic functions of the form $f(z)=\sum_{n=-\infty}^{+\infty} \lambda_n f_n(z)$, where $\lambda_n\rightarrow 0$, as $|n|\rightarrow \infty$. Define its norm as
\[
\|f\|_{c_{0,a,b}(\Omega_{r,R})}:=\sup_{n\in \mathbb{Z}} |\lambda_n|.
\]
\end{definition}
Note that when $b_{n}=0$ for all $n\in\mathbb{Z},$ we see that $\ell^p_{a,b}(\Omega_{r,R})$ is similar to a weighted $\ell^{p}(\mathbb{Z})$ space.  %Moreover, it is evident that the basis $\{f_n\}$ in $\ell^p_{a,b}(\mathtt{A}_{r,R})$ is equivalent to the standard ordered basis in $\ell^p(\mathbb{Z})$, $1\leq p < \infty$.

The next results follow in view of the previous proposition.

\begin{proposition}\label{duality}
The following hold.
\begin{itemize}
   \item[(i)] The space $\ell^p_{a,b}(\Omega_{r,R})$ is isometrically isomorphic to $\ell^p(\mathbb{Z})$, and also, the sequence $\{f_n\}$ forms a normalized Schauder basis in $\ell^p_{a,b}(\Omega_{r,R})$ and $c_{0,a,b}(\Omega_{r,R})$.
   \item[(ii)] The evaluation functionals $f\mapsto f(\zeta)$ are continuous on $\ell^p_{a,b}(\Omega_{r,R})$ and $c_{0,a,b}(\Omega_{r,R})$, where $r<|\zeta|<R$. 
   \item[(iii)] The dual of $\ell^p_{a,b}(\Omega_{r,R})$ is $\ell^q(\mathbb{Z})$, where $1\leq p<\infty$, and $1/p+1/q=1$. Similarly, the dual of $c_{0,a,b}(\Omega_{r,R})$ is $\ell^1_{a,b}(\Omega_{r,R})$.
   \item[(iv)] If $f=\sum_{n=-\infty}^{+\infty}\lambda_n f_n \in \ell^p_{a,b}(\Omega_{r,R})$ and $c_{0,a,b}(\Omega_{r,R})$, then its Laurent series expansion is 
   \[
   f(z)=\sum_{n=-\infty}^{+\infty} (\lambda_na_n+\lambda_{n-1}b_{n-1})z^n,
\]
where $z\in \Omega_{r,R}$.
    \end{itemize}
\end{proposition}
\begin{proof}
(i) This follows from the definition of the spaces $\ell^p_{a,b}(\Omega_{r,R})$ and $c_{0,a,b}(\Omega_{r,R})$. 

(ii) From the proof of the previous proposition, for $r<|\zeta|<R$, there exists a constant $M$ such that, 
   we have 
   \[
   |f(\zeta)|\leq \sum_{n=-\infty}^{+\infty} \lvert \lambda_{n}f_{n}(\zeta)\rvert\leq M\|f\|,
   \]
   for all $f=\sum_{n=-\infty}^{+\infty}\lambda_n f_n \in \ell^p_{a,b}(\Omega_{r,R})$ and $c_{0,a,b}(\Omega_{r,R})$. Hence, the evaluation functionals are continuous.
   
   (iii) We prove this result for $1<p<\infty$ only. Let $L$ be a bounded linear functional on $\ell^p_{a,b}(\Omega_{r,R})$. Consider the linear isomorphism $T:\ell^p(\mathbb{Z})\rightarrow \ell^p_{a,b}(\Omega_{r,R})$ given by $T(e_n)=f_n$, where $n\in \mathbb{Z}$. Then, $LT$ is a bounded linear functional on $\ell^p(\mathbb{Z})$, and hence, there exists $(y_n)\in \ell^q(\mathbb{Z})$ such that 
   \[
   (LT)(\sum_{n=-\infty}^{+\infty} \lambda_n e_n)=\sum_{n=-\infty}^{+\infty} \lambda_n y_n,
   \]
   that is, $L(\sum_{n=-\infty}^{+\infty} \lambda_n f_n)=\sum_{n=-\infty}^{+\infty} \lambda_n y_n$.
   
   (iv)
Fix $r<\rho<R$. For $n\in\mathbb{Z}$, consider the $n$-th Laurent coefficient of $f=\sum_{m=-\infty}^{+\infty}\lambda_m f_m \in \ell^{p}_{a,b}(\Omega_{r,R})$ and $c_{0,a,b}(\Omega_{r,R})$ defined by
$
k_n(f)
=\frac{1}{2\pi i}\int_{|z|=\rho}\frac{f(z)}{z^{n+1}}\,dz.
$
Since the series $\sum_{m=-\infty}^{+\infty}\lambda_m f_m$ converges uniformly on
the circle $|z|=\rho$, the linearity of the integral yields
\[
k_n(f)
=\sum_{m=-\infty}^{+\infty}\lambda_m
\frac{1}{2\pi i}\int_{|z|=\rho}\frac{f_m(z)}{z^{n+1}}\,dz.
\]
Since, for $m\in\mathbb{Z}$, we have $f_m(z)=(a_m+b_m z)z^m,$ we immediately get that
$
k_n(f)
=\lambda_n a_n+\lambda_{n-1}b_{\,n-1},
$
and hence, the proof is complete.
\end{proof}

\subsection{Continuity of coefficient functionals} Since each $f(z)$ in $\ell^p_{a,b}(\Omega_{r,R})$ is analytic on the annulus $\Omega_{r,R}$, we have the Laurent expansion $f(z)=\sum_{n=-\infty}^{+\infty} \widehat{f}(n)z^n$, $z\in \Omega_{r,R}$, where $\widehat{f}(n)$ is the $n$-th Laurent coefficient of $f(z)$. The $n$-th coordinate functional (or rather, a coefficient functional)  $k_n$ is given by 
\[
k_n(f)=\widehat{f}(n),
\]
and it is defined on $\ell^p_{a,b}(\Omega_{r,R})$ (and $c_{0,a,b}(\Omega_{r,R})$ similarly). Also, $f_n^*$ will denote coordinate functional with respect to the basis $\{f_n\}$ in $\ell^p_{a,b}(\Omega_{r,R})$, $1\leq p<\infty$, or $c_{0,a,b}(\Omega_{r,R})$, that is,
\[
f_n^*(f)=\lambda_n,\hskip .6cm f=\sum_{n=-\infty}^{+\infty}\lambda_n f_n.
\]

Below, we note that the continuity of the coefficient functionals follows from that of the evaluation functionals in a very general set up. Indeed, if $\mathcal{E}$ is a Banach space of analytic functions on an annulus $\Omega$ centered around $0$, having evaluation functionals bounded at each $z\in \Omega$. Let $ev_z$ denote the evaluation functional acting on $\mathcal{E}$. Since the  analyticity of 
    \[
    z\mapsto ev_z, ~z\in \Omega
    \]
    in the strong and norm operator topologies are equivalent, we see that $z\mapsto ev_z$ is an $\mathcal{E}^*$-valued norm-analytic function. It then admits a norm-convergent Laurent series:  
    \[
    ev_z=\sum_{n=-\infty}^{+\infty} L_n z^n,~ z\in \Omega,
    \]
    for some $L_n\in \mathcal{E}^*$, $n\in \mathbb{Z}$. On the other hand, if $f\in \mathcal{E}$, we have $f(z)=ev_z(f)=\sum_{n=-\infty}^{+\infty} L_n(f)z^n.$ Expanding $f(z)$ in its Laurent series, we must then have $L_n=k_n$ for all $n$. Thus, all coefficient functionals are continuous.

%\begin{proposition} Let $X$ be a Banach space of analytic functions on an annulus $\Omega$ in $\mathbb{C}$, having evaluation functionals bounded at every point in $\Omega$. The co-ordinate functional $k_n$ given by $f\mapsto \widehat{f}(n)$ is bounded for all $n\in \mathbb{Z}$, where $\widehat{f}(n)$ is the $n$-th Laurent coefficient of $f(z)$ with respect to the annulus $\Omega$.
%\end{proposition}\begin{proof}    Let us denote the function space $X$ as described in the proposition. To show that for each $n\in \mathbb{Z},$ the co-ordinate functional $k_{n}: X\to \mathbb{C}$ given by $k_{n}(f)=\widehat{f}(n)$ is bounded, we can use the Closed Graph Theorem. Since the Laurent coefficients of an analytic function depend analytically on the function, and since the evaluation functionals are bounded at every point in $\Omega$, therfore the Laurent coefficients are also bounded. Hence, the graph of $k_{n}$ is closed, and by the Closed Graph Theorem, the coordinate functional $k_{n}$ is bounded for all $n\in \mathbb{Z}.$\end{proof}

We now obtain certain general properties of the spaces $\ell^p_{a,b}(\Omega_{r,R})$ or $c_{0,a,b}(\Omega_{r,R})$. These, in particular include norm estimates for the coefficient functionals $k_n$.
\begin{proposition}\label{k_{n}}
 Let $k_n$ denote the coordinate functional $f\mapsto \widehat{f}(n)$, defined on $\ell^p_{a,b}(\Omega_{r,R})$ or $c_{0,a,b}(\Omega_{r,R})$, where $n\in \mathbb{Z}$. Then the following hold.
    \begin{itemize}
    \item[(i)] For $1<p<\infty$, $1/p+1/q=1$, we have
    \[
    \|k_n\|\leq (|a_n|^q+|b_{n-1}|^q)^{1/q}.
    \]
    For $p=1$, we have $\|k_n\|\leq \max\{|a_n|,|b_{n-1}|\}$. Also, for the case of $c_{0,a,b}(\Omega_{r,R})$, we have $\|k_{n}\|\leq |a_{n}|+|b_{n-1}|$.
    \item[(ii)]  $\{f_n,f_n^*\}$ is a biorthogonal system, and for $1<p<\infty$, $\{f_n^*\}$ is an unconditional Schauder basis for the dual space.
    \item[(iii)] We have 
\begin{center}  
$k_n=  a_nf_n^*+b_{n-1}f_{n-1}^*, ~~~~~~~~~ n\in \mathbb{Z}.$
\end{center}
   \end{itemize}
\end{proposition}
\begin{proof}(i) We provide the proof only for the case $1<p< \infty$. If $f \in \ell^{p}_{a,b}(\Omega_{r,R}),$ then $f(z)=\sum_{j=-\infty}^{+\infty} \lambda_{j}f_{j}(z)$, and $\|f\|_{\ell^p_{a,b}(\Omega_{r,R})}=\left(\sum_{j=-\infty}^{+\infty}\lvert \lambda_j\rvert^p\right)^{1/p}$. An arrangement into a Laurent series yields $f(z)=\sum_{j=-\infty}^{+\infty} (\lambda_ja_j+\lambda_{j-1}b_{j-1})z^j,$ and so, for all $n\in \mathbb{Z}$, we have 
$k_n(f)=\lambda_{n}a_{n}+\lambda_{n-1}b_{n-1}.$
It follows, by H\"{o}lder's inequality, that 
\begin{center}
$|k_n(f)|\leq \|f\|_{\ell^p_{a,b}(\Omega_{r,R})} (|a_n|^q+|b_{n-1}|^q)^{1/q}$
\end{center}
and hence $\|k_n\|\leq (|a_n|^q+|b_{n-1}|^q)^{1/q},$ for all $n\in \mathbb{Z}$.

(ii) This is immediate as $f_n^*(f_n)=1$, and $f_n^*(f_m)=0$ for $n\neq m$. Also, the equivalence of $\{f_n\}$ to the standard basis $\{e_n\}$ shows that $\{f_n^*\}$ and $\{e_n^*\}$ are equivalent bases which follow from reflexivity.

 (iii) For $f=\sum_{j=-\infty}^{+\infty} \lambda_{j}f_{j}\in \ell^{p}_{a,b}(\Omega_{r,R})$ or $c_{0,a,b}(\Omega_{r,R})$ and $n\in \mathbb{Z}$, it follows that $  ( a_nf_n^*+b_{n-1}f_{n-1}^{*})(f)=\lambda_{n}a_{n}+\lambda_{n-1}b_{n-1}=k_{n}(f),
 $ which gives the result. The proof is complete.
\end{proof}

We now obtain conditions such that the Laurent monomials $z^{\nu}$ belong to the spaces $\ell^p_{a,b}(\Omega_{r,R})$ and $c_{0,a,b}(\Omega_{r,R})$.  The norm estimates, stated below, will later be used in deducing the dynamical properties of weighted shifts.

\begin{proposition}\label{estimate}
 Consider the space $\ell^{p}_{{a},{b}}(\Omega_{r,R}),$ $1\leq p<\infty$. Then, $z^{\nu}\in \ell^p_{a,b}(\Omega_{r,R})$ if and only if
\[
\|z^{\nu}\|^{p}
=\frac{1}{|a_{\nu}|^{p}}\left(1+\sum_{k=0}^{+\infty}
\prod_{j=0}^{k}
\left|
\frac{b_{\nu+j}}{a_{\nu+j+1}}
\right|^{p}\right)
<\infty, ~\nu \in \mathbb{Z}.
\]
Under the stronger assumption that
$ \limsup_{ n\rightarrow +\infty} \left\lvert b_n/a_{n+1}\right\rvert<1,$ then there exists a constant $M_{1}>0$ such that
   \begin{equation}\label{Estimates-q}
   \lVert z^{\nu}\rVert\leq M_{1}\frac{1}{|a_{\nu}|},
   \end{equation}
   for all $\nu \in \mathbb{Z}.$ 
\end{proposition}
\begin{proof}
For a fixed $\nu\in\mathbb{Z},$  by the basis expansion in $\ell^{p}_{{a},{b}}(\Omega_{r,R})$ and a rearrangement into Laurent series, we can find some $\{\lambda_{n}\}_{n=-\infty}^{+\infty}\in \ell^p(\mathbb{Z})$ such that
 \begin{equation} \label{1}
    z^{\nu}= \sum_{n=-\infty }^{+\infty} \lambda_{n}f_{n}(z)=\sum_{n=-\infty}^{+\infty}\left(\lambda_{n}a_{n}+\lambda_{n-1}b_{n-1}\right)z^{n},
     \end{equation}
  for all $z$ in an annulus $\Omega_{r,R}$. Now equating the coefficients of like-powers, we have
\[
\lambda_n =
\begin{cases} 
(-1)^{\,n-\nu}\,\lambda_{\nu}
\frac{b_{\nu} b_{\nu+1} \cdots b_{n-1}}{a_{\nu+1} a_{\nu+2} \cdots a_{n}},
& \quad n \ge \nu+1, \\[1em]
\lambda_{\nu},
& \quad n = \nu, \\[0.5em]
\frac{1-\lambda_{\nu}a_{\nu}}{b_{\nu-1}},& \quad n = \nu-1, \\[1em]
(-1)^{\,\nu-n-1}\,\lambda_{\nu-1}
\frac{a_{n+1} a_{n+2} \cdots a_{\nu-1}}{b_{n} b_{n+1} \cdots b_{\nu-2}},
& \quad n \le \nu-2.
\end{cases}
\]
Since we can only accept solutions $\{\lambda_{n}\}_{n=-\infty}^{+\infty}\in \ell^p(\mathbb{Z}),$ this requires that 
\[
|\lambda_{\nu-1}|^{p}\sum_{n\le \nu-2}
\left|\frac{a_{n+1} a_{n+2} \cdots a_{\nu-1}}{b_{n} b_{n+1} \cdots b_{\nu-2}}\right|^{p}<\infty~~~~\text{and}~~~~|\lambda_{\nu}|^{p}\sum_{n\ge \nu+1}
\left|\frac{b_{\nu} b_{\nu+1} \cdots b_{n-1}}{a_{\nu+1} a_{\nu+2} \cdots a_{n}}\right|^{p}<\infty.
\]
But, for any $\nu$ (positive or negative), these conditions are equivalent, respectively, to 
\[
|\lambda_{\nu-1}|^{p}\sum_{n\le-1}
\left|\frac{a_{n+1} a_{n+2} \cdots a_{0}}{b_{n} b_{n+1} \cdots b_{-1}}\right|^{p}<\infty~~~~\text{and}~~~~|\lambda_{\nu}|^{p}\sum_{n\ge 1}
\left|\frac{b_{0} b_{1} \cdots b_{n-1}}{a_{1} a_{2} \cdots a_{n}}\right|^{p}<\infty.
\]
In view of condition $\sup_{n\leq -1} \left |\frac{a_{n+1}a_{n+2}\cdots a_0}{b_nb_{n+1}\cdots b_{-1}}\right|=\infty,$ we necessarily have $\lambda_{\nu-1}=0$ and $\lambda_{\nu}=\frac{1}{a_{\nu}}\neq 0.$
Thus, we get, for $\nu\in \mathbb{Z}$,
\begin{equation}\label{monomial-expansion}
    z^{\nu}=\frac{1}{a_{\nu}}\left(f_{\nu}+\sum_{j=1}^{\infty}(-1)^{j}\frac{b_{\nu}b_{\nu+1}\cdots b_{\nu+j-1}}{a_{\nu+1}a_{\nu+2}\cdots a_{\nu+j}}f_{\nu+j}\right),
\end{equation}
and also, we have
\[
\lVert z^{\nu}\rVert^{p}=\frac{1}{|a_{\nu}|^{p}}\left(1+\sum_{k=0}^{+\infty}
\prod_{j=0}^{k}\left|\frac{b_{\nu+j}}{a_{\nu+j+1}}\right|^{p}\right).
\]
On the other hand, the assumption $\limsup_{ n\rightarrow +\infty} \left\lvert b_n/a_{n+1}\right\rvert<1,$ we can find $r<1$ and an integer $N_1\geq 1,$ such that $\big|b_{n}/a_{n+1}\big|<r,$ for all $n\geq N_1.$ Thus, we have $\lVert z^{\nu}\rVert^{p}\leq \frac{1}{|a_{\nu}|^{p}}\big(\sum_{j=0}^{+\infty}r^{pj}\big),$ for every $\nu\geq N_{1}$, which completes the proof.
 \end{proof}
The $c_{0,a,b}(\Omega_{r,R})$ analogue of the above result is as follows, and its proof is omitted.
 
 \begin{proposition}
     Consider the space $c_{0,a,b}(\Omega_{r,R})$. Then, for $\nu \in \mathbb{Z}$, $z^{\nu}\in c_{0,a,b}(\Omega_{r,R})$ if and only if
\begin{equation*}
 \lVert  z^{\nu}  \rVert_{c_{0,a,b}(\Omega_{r,R})} = \max \left \{\frac{1}{|a_\nu|},\frac{1}{\lvert a_{\nu}\rvert}\sup_{j\geq 0}
\left|
\frac{b_{\nu}b_{\nu+1}\cdots b_{\nu+j}}{a_{\nu+1}a_{\nu+2}\cdots a_{\nu+j+1}}
\right|\right \}<\infty.
 \end{equation*}
 In addition, if $\limsup_{ n\rightarrow +\infty} \left\lvert b_n/a_{n+1}\right\rvert<1,$  then there exists a constant $M_{2}>0$ such that
   \begin{equation*}
   \lVert z^{\nu}\rVert\leq M_{2}\frac{1}{|a_{\nu}|},
   \end{equation*}
   for all $\nu \in \mathbb{Z}.$ 
 \end{proposition}
\section{Weighted backward shift $B_w$: matrix forms, and consequences}

  In this section we introduce \textit{the bilateral weighted backward shift} operator $B_w$ on the spaces $\ell^{p}_{a,b}(\Omega_{r,R})$ and $c_{0,a,b}(\Omega_{r,R})$. We find boundedness conditions for $B_w$. The essential spectrum of $B_w$ is obtained via a compact perturbation result. 

 Let $w:=\{w_n\}_{n=-\infty}^{+\infty}$ be a complex sequence. The operator $B_w$ is defined by 
 \[
 B_w(\sum_{n=-\infty}^{+\infty}\lambda_nz^n):=\sum_{n=-\infty}^{+\infty}\lambda_nw_nz^{n-1}.
 \]
Thus, \(B_w(z^n)=w_n z^{\,n-1}\) for all \(n\in\mathbb{Z}\), provided the underlying space contains all Laurent polynomials.

To derive some necessary and sufficient conditions for the weighted backward shift $B_w$ to be bounded, we compute the matrix representation of $B_w$ acting on $\ell^{p}_{a,b}(\Omega_{r,R})$ with respect to the ordered normalized Schauder basis
\begin{center}  
$f_n(z)= (a_n+b_{n}z)z^{n}, \hskip.5cm n\in \mathbb{Z}$
\end{center} 
in $\ell^{p}_{a,b}(\Omega_{r,R})$. Then, the boundedness of $B_w$ is equivalent to that of its matrix operator acting on $\ell^p(\mathbb{Z})$. Indeed, the (bilateral) matrix $[B_w]$ is given by
   \[   
  [B_w]=
  \begin{bmatrix}
  \ddots&  \ddots & \ddots & \ddots & \ddots&\ddots&\ddots\\
  \ddots& 0 & 0& 0&\ddots& \ddots&\ddots\\
  \ddots&\frac{w_{-2}a_{-2}}{a_{-3}} & 0& 0&0& \ddots&\ddots\\
\ddots& c_{-2}& \frac{w_{-1}a_{-1}}{a_{-2}} & 0&0 & 0 &\ddots\\
\ddots& -c_{-2}\frac{b_{-2}}{a_{-1}}& c_{-1} & \frac{w_{0}a_{0}}{a_{-1}} & 0&0&\ddots\\
  \ddots & c_{-2}\frac{b_{-2}b_{-1}}{a_{-1}a_{0}}&-c_{-1}\frac{b_{-1}}{a_{0}}&\boxed{\boldsymbol{c_{0}}} &\frac{w_{1}a_{1}}{a_{0}} &0&\ddots \\
	\ddots & -c_{-2}\frac{b_{-2}b_{-1}b_{0}}{a_{-1}a_{0}a_{1}}&c_{-1}\frac{b_{-1}b_{0}}{a_{0}a_{1}}&-c_{0}\frac{b_{0}}{a_{1}} & c_{1}&\frac{w_{2}a_{2}}{a_{1}}& \ddots\\
	\ddots & c_{-2}\frac{b_{-2}b_{-1}b_{0}b_{1}}{a_{-1}a_{0}a_{1}a_{2}} & -c_{-1}\frac{b_{-1}b_{0}b_{1}}{a_{0}a_{1}a_{2}}&c_{0}\frac{b_{0}b_{1}}{a_{1}a_{2}}&-c_{1}\frac{b_{1}}{a_{2}} &c_{2}&\ddots\\
	\ddots &  \ddots &c_{-1}\frac{b_{-1}b_{0}b_{1}b_{2}}{a_{0}a_{1}a_{2}a_{3}}&-c_{0}\frac{b_{0}b_{1}b_{2}}{a_{1}a_{2}a_{3}}&c_{1}\frac{b_{1}b_{2}}{a_{2}a_{3}}&-c_{2}\frac{b_{2}}{a_{3}} &\ddots\\
 \ddots &   \ddots & \ddots&c_{0}\frac{b_{0}b_{1}b_{2}b_{3}}{a_{1}a_{2}a_{3}a_{4}}&-c_{1}\frac{b_{1}b_{2}b_{3}}{a_{2}a_{3}a_{4}}&c_{2}\frac{b_{2}b_{3}}{a_{3}a_{4}} &\ddots\\
	  \ddots & \ddots & \ddots&\ddots&\ddots &\ddots&\ddots 
	\end{bmatrix}.
 \]
Here, $c_n$ is the element given below. Also, in this matrix, the box bold element $\boxed{\boldsymbol{c_{0}}} $ is at the $(0,0)$-th entry. The above matrix form can be verified as follows.\\We will use the basis expansions of $z^n$ given in the equation \eqref{monomial-expansion}. We have for all $n\in\mathbb{Z}$,
\[
B_{w}(f_{n})(z)=w_{n}a_{n}z^{n-1}+w_{n+1}b_{n}z^{n}=\frac{w_{n}a_n}{a_{n-1}}f_{{n-1}} + \left(\frac{w_{n+1}b_{n}}{a_{n}}-\frac{w_{n}b_{n-1}}{a_{n-1}}\right)a_{n} z^{n},
\]
by substituting $z^{n-1}=\frac{1}{a_{n-1}}(f_{n-1}-b_{n-1}z^{n})$. Set
\[
c_{n}:=w_{n+1}\frac{b_{n}}{a_{n}}-w_{n}\frac{b_{n-1}}{a_{n-1}},~~~~n\in \mathbb{Z}.
\]
In the following, we now get the $n$-th column of the matrix of $B_w$, for $n\in \mathbb{Z}$, by using the expansion of $z^{n}$ in terms of the basis elements $f_n$'s:
\[
B_{w}(f_{n})(z)=\frac{w_{n}a_n}{a_{n-1}}f_{{n-1}}+c_{n}f_{n}+c_{n}\sum_{j=1}^{+\infty}(-1)^{j}\left(\Mprod_{k=0}^{j-1}\frac{b_{n+k}}{a_{n+k+1}}\right)f_{n+j}.
\]

 From the matrix representation, below we obtain sufficient and necessary conditions for $B_w$ to be a bounded operator.
    
   \begin{theorem}\label{decomposition1}
  Assume that $z^n\in \ell^p_{a,b}(\Omega_{r,R})$ for all $n\in \mathbb{Z}$, and 
  \begin{equation}\label{rr}
  \sup_{n\in \mathbb{Z}}~\left\lvert\frac{ w_{n+1}a_{n+1}}{a_{n}}\right\rvert<\infty.
  \end{equation}
  \begin{itemize}
  \item[(i)] If 
  \begin{equation}\label{ss}
  \sum_{i=1}^{+\infty}\sup_{n\in \mathbb{Z}}~~\left\lvert c_{n}\frac{b_{n}b_{n+1}\cdots b_{n+i-1}}{a_{n+1}a_{n+2}\cdots a_{n+i}}\right\rvert<\infty, ~\text{and}~ \sup_{n\in \mathbb{Z}}|c_n|<\infty,
  \end{equation}
  then $B_{w}$ is bounded on $\ell^{p}_{a,b}(\Omega_{r,R})$.
\item[(ii)] If 
\begin{equation}\label{tt}
\limsup_{ n\rightarrow +\infty} \left\lvert b_n/a_{n+1}\right\rvert<1,
\end{equation}
then the conditions in \textnormal{(i)} are satisfied and consequently, $B_w$ is bounded.
   \end{itemize}
   A similar result holds for $c_{0,a,b}(\Omega_{r,R}).$
   \end{theorem}
   \begin{proof}
(i) It suffices to show that the above infinite matrix acts as a bounded operator on $\ell^p(\mathbb{Z})$. To this end, write the matrix of $B_{w}$ as a formal series of infinite matrices as follows:
\[
 [B_{w}]=T_{-1}+D+T_{1}+T_{2}+\cdots,
 \]
where $D$ is the matrix of the diagonal operator 
\[
\text{diag}~ (\cdots,c_{-2},c_{-1},\boxed{\boldsymbol{c_{0}}},c_{1},c_{2},\cdots)
\]
on $\ell^{p}(\mathbb{Z}).$ Also, the matrix $T_i$ is obtained by deleting all the entries of $[B_{w}]$, except those at the $i$-th subdiagonal, where $i\geq 1$. Observe that $T_i$ is the matrix of suitable powers of a bilateral weighted forward shift.  On the other hand, the matrix $T_{-1}$ is obtained by deleting all the entries of $[B_{w}]$, except those at the first superdiagonal. Note that $T_{- 1}$ is the matrix of the standard weighted bilateral backward shift $T_{-1}(e_n)\mapsto \gamma_n e_{n-1}$ on $\ell^p(\mathbb{Z})$, $n\in \mathbb{Z}$, having weights 
\[
\gamma_n= w_{n+1}\frac{a_{n+1}}{a_{n}}. 
\]
For each $n \in \mathbb{Z},$ we have
\[
c_{n}=w_{n+1}\frac{b_{n}}{a_{n}}-w_{n}\frac{b_{n-1}}{a_{n-1}}=\frac{w_{n+1}a_{n+1}}{a_{n}}\frac{b_{n}}{a_{n+1}}-w_{n}\frac{a_n}{a_{n-1}}\frac{b_{n-1}}{a_{n}}.
\]
By the assumptions, it follows that $\{c_{n}\}_{n\in \mathbb{Z}}$ is a bounded sequence. Moreover, we formally have
\begin{center}
$\lVert D \rVert=  \sup_{n\in \mathbb{Z}}~~\lvert c_{n}\rvert, $
\end{center}
and for all $i\geq 1$
\[
 \lVert T_{i}\rVert=\sup_{n\in\mathbb{Z}}~~\left\lvert c_{n}\frac{b_{n}b_{n+1}\cdots b_{n+i-1}}{a_{n+1}a_{n+2}\cdots a_{n+i}}\right\rvert.
\]
(ii) Since
$ \limsup_{ n\rightarrow +\infty} \left\lvert b_n/a_{n+1}\right\rvert<1,$
one gets $n_{0}\in\mathbb{N}$ and $r<1$ such that $|b_n/a_{n+1}|<r$ for all $n\geq n_{0}$. Set
\[
M:=\sup_{n\in \mathbb{Z}}\left\{\left\lvert\frac{b_{n}}{a_{n+1}}\right\rvert,|c_{n}|\right\}.
\]
Then $\lVert T_{i}\rVert\leq M^{i+1}$ for all $i=1,\cdots ,n_{0}$ and $\lVert T_{i}\rVert\leq M^{n_{0}+1}r^{i-n_{0}}$ for all $i> n_{0},$ from which it follows that
\begin{eqnarray*}
    \sum_{i=n_{0}+1}^{+\infty}\lVert T_{i}\rVert=M^{n_{0}+1}\sum_{i=n_{0}+1}^{+\infty}r^{i-n_{0}}=M^{n_{0}+1}\frac{r}{1-r}.
\end{eqnarray*}
Hence, the series
$\sum_{i=1}^{+\infty}\lVert T_{i}\rVert$ is convergent. Consequently, the matrix $[B_w]$ acts as a bounded operator on $\ell^p(\mathbb{Z})$. This completes the proof of the theorem. 
   \end{proof}

 A necessary condition for $B_w$ to be bounded on $\ell^p_{a,b}(\Omega_{r,R})$ and $c_{0,a,b}(\Omega_{r,R})$ is stated below.

 \begin{proposition}
     If $B_w$ is a bounded operator on any of the spaces $\ell^p_{a,b}(\Omega_{r,R})$ and $c_{0,a,b}(\Omega_{r,R})$, then $\{\frac{ w_{n+1}a_{n+1}}{a_{n}}\}_{n\in\mathbb{Z}}$ and $\{c_{n}\}_{n\in\mathbb{Z}}$ are bounded, assuming that these spaces contain all Laurent polynomials.
 \end{proposition}
 \begin{proof}
     If $B_w$ is bounded on $\ell^p_{a,b}(\Omega_{r,R})$, then its matrix $[B_w]$ acts as a bounded operator on the space $\ell^p(\mathbb{Z})$.
     Hence, 
     \[
     \sup_{n\in \mathbb{Z}}||[B_w](e_n)||_{\ell^p(\mathbb{Z})}<\infty,
     \]
     where $e_n$ is the standard unit vector in $\ell^p(\mathbb{Z})$. On the other hand, $[B_w](e_n)$ is the $n$-th column of the bilateral matrix $[B_w]$. With the observation that
     \[
     \left|\frac{ w_{n+1}a_{n+1}}{a_{n}}\right|^{p}+|c_{n}|^{p}\leq \|[B_w](e_n)\|_{\ell^p(\mathbb{Z})}^{p}, \hskip .2cm n\in \mathbb{Z},
     \]
     we immediately get $\sup_{n\in \mathbb{Z}}|\frac{ w_{n+1}a_{n+1}}{a_{n}}|<\infty$ and $\sup_{n\in\mathbb{Z}}|c_{n}|<\infty.$ The case of $c_{0,a,b}(\Omega_{r,R})$ is similar.
 \end{proof}

Before we study the continuity of $B_w$ on $\ell^{p}_{a,b}(\Omega_{r,R})$ and $c_{0,a,b}(\Omega_{r,R})$, we provide an important ``weighted forward shift" property of the adjoint $B_w^*$ in a general set up.

\begin{proposition}\label{act-like}
Let $\mathcal{E}$ be a Banach space of analytic functions defined on some annulus $r<|z|<R$, having continuous evaluation functionals. Let $B_{w}$ be bounded on $\mathcal{E}$. Then, 
\begin{center}
$B_{w}^{*}(k_{n}) =w_{n+1}k_{n+1},$
\end{center} for all $n\in \mathbb{Z},$ where $k_n$ denotes the $n$-th coordinate functional.
\end{proposition}
\begin{proof}
  For $f(z)=\sum_{j=-\infty}^{+\infty}\lambda_{j}z^{j}$ in $\mathcal{E}$, we have $k_{n+1}(f)=\lambda_{n+1}.$  Also, we have 
  \begin{center}
  $B_{w}^{*}\big(k_{n}\big)(f)=\big(k_{n}\circ B_{w}\big)(f)=w_{n+1}\lambda_{n+1}.$
   \end{center}
  The required result follows.
\end{proof}
In terms of $k_n$, a general necessary condition can be obtained for $B_w$ to be bounded, as follows:
\begin{proposition}
    Suppose $\mathcal{E}$ is a Banach space of analytic functions defined on an annulus $\Omega_{r,R}$, having continuous evaluation functionals. Suppose that $k_n\neq 0$ for all $n\in \mathbb{Z}$. If $B_w$ is a bounded operator on $\mathcal{E}$, then $\sup ~\Big\{|w_{n+1}|\frac{\|k_{n+1}\|}{\|k_n\|}:~{n\in \mathbb{Z}}\Big\}<\infty.$
\end{proposition}
\begin{proof}
Consider the vectors of unit norm, $u_n:= \frac{k_n}{\|k_n\|}, ~n\in \mathbb{Z}$.
Since $B_w$ is a bounded operator, we get that $\{B_w^*(u_n):n\in \mathbb{Z}\}$ is a norm-bounded set in $\mathcal{E}^*$. This, along with Proposition \ref{act-like}, implies the result.
\end{proof}

\subsection{$B_w$ as compact + usual shifts, and the essential spectrum} As a consequence of the matrix representation of $B_w$, we provide a sufficient condition so that the operator becomes a compact perturbation of a usual weighted shift. A consequence is the description of the essential spectrum $\sigma_e(B_w)$ of the weighted shift $B_w$ acting on $\ell^p_{a,b}(\Omega_{r,R})$, for $1<p<\infty$. For a bounded operator $T$ on a complex Banach space $X$, the essential spectrum is denoted and  defined by 
\begin{center}
$\sigma_e(T):=\{\lambda \in \mathbb{C}:$ ~ $T-\lambda I $~is~not~Fredholm$\}$,
\end{center}
cf.  Shields \cite{Shields}. See, also, Bayart and Matheron \cite{Bayart-Matheron} and Douglas \cite{Douglas} for an introduction to essential spectrum of operators including weighted shifts.
 
\begin{theorem}\label{prop}
If $B_w$ is bounded on $\ell^p_{a,b}(\Omega_{r,R})$, and
\[
\lim_{ |n| \rightarrow +\infty} \left\lvert w_{n+1}\frac{b_{n}}{a_{n}}-w_{n}\frac{b_{n-1}}{a_{n-1}}\right\rvert=0 ,
\]
then the operator $B_{w}$ on $\ell^{p}_{{a},{b}}(\Omega_{r,R})$ is similar to $B_{\alpha}+K$ for some compact operator $K$ and a bilateral weighted backward shift $B_{\alpha}$ acting on the sequence space $\ell^p(\mathbb{Z})$, where the weight sequence $\alpha=(\alpha_n)$ is given by 
\[
\alpha_n=\frac{ w_{n+1}a_{n+1}}{a_{n}}, 
\]
for all $n\in \mathbb{Z}.$ Consequently, $\sigma_e(B_w)=\sigma_e(B_{\alpha})$, which is an annulus.
 \end{theorem}

 \begin{proof}
     This result follows from the above matrix representation of $B_{w}$. Consider the similarity operator $V:\ell^{p}_{a,b}(\Omega_{r,R})\rightarrow \ell^p(\mathbb{Z})$ given by $V\big(\sum_{n=-\infty}^{+\infty} \lambda_n f_n\big)=\sum_{n=-\infty}^{+\infty}\lambda_ne_n$, where $\{e_n\}_{n\in \mathbb{Z}}$ is the standard basis in $\ell^p(\mathbb{Z})$. Now, from the proof of Theorem \ref{decomposition1} we understand that $B_{w}$ acting on $\ell^{p}_{a,b}(\Omega_{r,R})$ is similar via $V$ to the operator given by the series (which is convergent absolutely with respect to the operator norm)
     $       T_{-1}+D+\sum_{i=1}^{+\infty}T_i.
     $
     Further, by our assumptions, the operators $D$ and $T_i~~(i\geq 1)$ are compact on $\ell^p(\mathbb{Z})$ as the entries in each of $D$ and $T_i$ converge to $0$. Also, the operator $T_{-1}$ is the weighted bilateral backward shift on $\ell^p(\mathbb{Z})$ with weights $\alpha=\{\alpha_n\}_{n=-\infty}^{+\infty}$, as in the statement of this theorem. We now take $B_{\alpha}:=T_{-1}$, \text{and} $K:=D+\displaystyle\sum_{i=1}^{+\infty}T_i$, required as in the theorem, since $K$ is a compact operator on $\ell^p(\mathbb{Z})$. 

Since the essential spectrum is invariant under compact perturbations, we have $\sigma_e(B_w)=\sigma_e(B_{\alpha}+K)=\sigma_e(B_{\alpha})$. The essential spectrum of a bilateral weighted shift is an annulus, cf. Shields \cite{Shields}.
 \end{proof}

\section{Hypercyclicity, supercyclicity, and chaos}

In this section, we establish the necessary and sufficient conditions for 
\(B_w\) to be hypercyclic, mixing, supercyclic, or chaotic 
on \(\ell^{p}_{a,b}(\Omega_{r,R})\). Analogous results hold for \(c_{0,a,b}(\Omega_{r,R})\) and therefore we omit the details for the latter case. We will need the following standard criteria.

\begin{theorem}\label{thm-hypc} $\emph{\textsf{(Gethner-Shapiro Criterion~\cite{Gethner-Shapiro})}}$
Let $T$ be a bounded operator on a separable Banach space $X$, and let $D$ be a dense subset of $X$. If $\{n_k\} \subseteq \mathbb{N}$ is a strictly
increasing sequence and $S:D\rightarrow D$ is a map such that, for each
$x\in D$, 
\[
\lim_{k\rightarrow \infty} T^{n_k}x=0=\lim_{k\rightarrow \infty}S^{n_k}x,
\]
and $TSx=x$, then $T$ is hypercyclic. In addition, if $n_k=k$ for all $k\geq 1$, then $T$ is mixing on $X$.
\end{theorem}

\begin{theorem} \emph{\textsf{(Supercyclicity Criterion \cite{Montes1})}}
For a bounded operator $T$ on a separable Banach space $X$, if there exist a dense subset $D$, a sequence $(n_k)$ of natural numbers, and a map $S:D\rightarrow D$ such that, for all $x,y\in D$,  
\begin{center}
$\lim_{k\rightarrow \infty}\|T^{n_k}x\|\|S^{n_k}y\|=0,$
\end{center}
and $TSx=x$, then $T$ is supercyclic.
\end{theorem}

We provide characterizations for the aforementioned dynamical properties of $B_w$ when the basis in the following form. Indeed, if we consider 
 \begin{center}  
$f_n(z)=
\begin{cases}
   (a_n+b_{n}z)z^{n}, \hskip.5cm n\geq 0,\\ a_{n} z^{n},\hskip 1.9cm n\leq -1,
\end{cases}$
\end{center} 
as the ordered normalized Schauder basis in $\ell^{p}_{a,b}(\Omega_{r,R})$, then the (bilateral) matrix $[B_w]$ of the operator $B_{w}$ is given by
 \[   
  [B_w]=
  \begin{bmatrix}
  \ddots&  \ddots & \ddots & \ddots & \ddots&\ddots&\ddots\\
  \ddots& 0 & 0& 0&\ddots& \ddots&\ddots\\
  \ddots&\frac{w_{-2}a_{-2}}{a_{-3}} & 0& 0&0& \ddots&\ddots\\
\ddots& 0& \frac{w_{-1}a_{-1}}{a_{-2}} & 0&0 & 0 &\ddots\\
\ddots&0& 0 & \frac{w_{0}a_{0}}{a_{-1}} & 0&0&\ddots\\
  \ddots & 0&0&\boxed{\boldsymbol{c_{0}}} &\frac{w_{1}a_{1}}{a_{0}} &0&\ddots \\
	\ddots & 0&0&-c_{0}\frac{b_{0}}{a_{1}} & c_{1}&\frac{w_{2}a_{2}}{a_{1}}& \ddots\\
	\ddots & 0 & 0&c_{0}\frac{b_{0}b_{1}}{a_{1}a_{2}}&-c_{1}\frac{b_{1}}{a_{2}} &c_{2}&\ddots\\
	\ddots &  \ddots &0&-c_{0}\frac{b_{0}b_{1}b_{2}}{a_{1}a_{2}a_{3}}&c_{1}\frac{b_{1}b_{2}}{a_{2}a_{3}}&-c_{2}\frac{b_{2}}{a_{3}} &\ddots\\
 \ddots &   \ddots & \ddots&c_{0}\frac{b_{0}b_{1}b_{2}b_{3}}{a_{1}a_{2}a_{3}a_{4}}&-c_{1}\frac{b_{1}b_{2}b_{3}}{a_{2}a_{3}a_{4}}&c_{2}\frac{b_{2}b_{3}}{a_{3}a_{4}} &\ddots\\
	  \ddots & \ddots & \ddots&\ddots&\ddots &\ddots&\ddots 
	\end{bmatrix}.
 \]

To obtain necessary conditions for $B_w$ to be hypercyclic or supercyclic, we require the matrix representation of $B_{w}^{\nu}$, which is computed below.

For $\nu \geq 1$ and $n \in \mathbb{Z}$, set 
\[
C_{n,\nu}=
\begin{cases}
w_{n+1}\cdots w_{n-\nu+2}\frac{b_{n}}{a_{n-\nu+1}}
 - w_{n}\cdots w_{n-\nu+1}\frac{a_{n}b_{n-\nu}}{a_{n-\nu} a_{n-\nu+1}}, 
 & \text{if } n\ge \nu,\\[0.4em]

w_{n+1}\cdots w_{n-\nu+2}\,\frac{b_{n}}{a_{n-\nu+1}}, 
 & \text{if } 0\leq n \leq \nu-1,\\[0.4em]

0, & \text{if } n\leq -1,
\end{cases}
\]
and $A_{n,\nu}
 := w_{n}\cdots w_{n-\nu+1}\,\frac{a_{n}}{a_{n-\nu}}.$\\ By the action of $B_w^\nu$ on the normalized Schauder basis $\{f_n\}$ of $\ell^{p}_{a,b}(\Omega_{r,R})$, we obtain the following formulas.\\
 For all $n>0$,
$
B_w^{\nu}(f_{-n})(z)
  = w_{-n}\cdots w_{-n-\nu+1}\, a_{-n}\, z^{-n-\nu} =A_{-n,\nu}\, f_{-n-\nu}.$\\
When $0 \le n \le \nu-2$, we have
\[
\begin{aligned}
B_w^{\nu}(f_{n})(z)
  &= w_{n}\cdots w_{n-\nu+1}\, a_{n}\, z^{n-\nu}
     \;+\;
     w_{n+1}\cdots w_{n-\nu+2}\, b_{n}\, z^{n-\nu+1} \\[6pt]
  &= \frac{w_{n}\cdots w_{n-\nu+1}\, a_{n}}{a_{n-\nu}}\,
       f_{n-\nu}
     \;+\;
     \left(
       \frac{w_{n+1}\cdots w_{n-\nu+2}\, b_{n}}{a_{n-\nu+1}}
     \right)
     a_{n-\nu+1}\, z^{n-\nu+1} \\[6pt]
  &= A_{n,\nu}\, f_{n-\nu}
     \;+\; C_{n,\nu}\, f_{n-\nu+1}.
\end{aligned}
\]
For the remaining case $n \ge \nu-1$, we obtain
\[
\begin{aligned}
B_w^{\nu}(f_{n})(z)
  &= w_{n}\cdots w_{n-\nu+1}\, a_{n}\, z^{n-\nu}+
     w_{n+1}\cdots w_{n-\nu+2}\, b_{n}\, z^{n-\nu+1} \\[6pt]
  &= \frac{w_{n}\cdots w_{n-\nu+1} a_{n}}{a_{n-\nu}}
       f_{n-\nu}+\left(w_{n+1}\cdots w_{n-\nu+2} b_{n}-\frac{w_{n}\cdots w_{n-\nu+1}a_{n} b_{n-\nu}}{a_{n-\nu}} \right)
     z^{n-\nu+1} \\[6pt]
  &= A_{n,\nu}\, f_{n-\nu}
     \;+\;
     C_{n,\nu}\, a_{n-\nu+1}\, z^{n-\nu+1}.
\end{aligned}
\]
Using the expansion of $z^{\,n-\nu+1}$ from equation \eqref{monomial-expansion}, we obtain
\[
B_{w}^{\nu}(f_{n})(z)
 = A_{n,\nu}\, f_{n-\nu}
   \;+\;
   C_{n,\nu}
   \left(
     f_{n-\nu+1}
     + \sum_{j=1}^{\infty}
         (-1)^{j}
         \left(
           \prod_{k=0}^{j-1}
           \frac{b_{n-\nu+1+k}}{a_{n-\nu+2+k}}
         \right)
         f_{n-\nu+1+j}
   \right).
\]
Therefore, the matrix of $B_{w}^{\nu}$ with respect to the normalized Schauder basis
$\{ f_n \}_{n \in \mathbb{Z}}$ is given by

 \begin{equation}\label{matrix}
  [B_w^{\nu}]=
  \begin{bmatrix}
  \ddots&  \ddots & \ddots & \ddots & \ddots&\ddots&\ddots&\ddots&\ddots\\
   \ddots & A_{-1,\nu}&0&0&\cdots&0&0&0&\ddots\\
   \ddots & 0& A_{0,\nu}&0&\cdots&0&0&0&\ddots\\
  \ddots & 0& C_{0,\nu}&A_{1,\nu}&\cdots&0&0&0&\ddots\\
  \ddots & 0& 0&C_{1,\nu}& \cdots&0&0&0&\ddots\\
  \ddots & \vdots& \vdots&\vdots& \cdots&\vdots&\vdots&\vdots&\ddots\\
\ddots& 0 & 0&0 & \cdots&A_{\nu-2,\nu}&0&0 &\ddots\\
\ddots& 0 & 0 & 0&\cdots&C_{\nu-2,\nu}&A_{\nu-1,\nu}&0&\ddots\\
  \ddots &0&\boxed{\boldsymbol{0}} &0&\cdots &0&C_{\nu-1,\nu}& A_{\nu,\nu}&\ddots\\
	\ddots &0&0& 0&\cdots&0&-C_{\nu-1,\nu}\frac{b_{0}}{a_{1}}& C_{\nu,\nu}&\ddots\\
	\ddots  & 0&0&0 &\cdots&0&C_{\nu-1,\nu}\frac{b_{0}b_{1}}{a_{1}a_{2}}&-C_{\nu,\nu}\frac{b_{1}}{a_{2}}&\ddots\\
 \ddots  & \ddots&0&0&\cdots&0&-C_{\nu-1,\nu}\frac{b_{0}b_{1}b_{2}}{a_{1}a_{2}a_{3}}&C_{\nu,\nu}\frac{b_{1}b_{2}}{a_{2}a_{3}}&\ddots\\
	  \ddots & \ddots&\ddots&\ddots &\ddots&\ddots&\ddots&\ddots&\ddots 
	\end{bmatrix}.
 \end{equation}
Then, for $u=(\lambda_{j})_{j\in \mathbb{Z}}\in \ell^{p}(\mathbb{Z}),$ we get
\begin{equation}\label{cc}
   [B_{w}^{\nu}]u=\sum_{j=-\infty}^{+\infty} \alpha_{j,\nu} ~ e_{j},
\end{equation}
where $\{e_{j}\}_{j\in \mathbb{Z}}$ is the standard basis in $\ell^p(\mathbb{Z})$, and 
\[
\alpha_{j,\nu} =
\begin{cases}
A_{\nu+j,\nu}\,\lambda_{\nu+j}, 
& \text{if } j \leq -\nu,\\[4pt]
C_{\nu+j-1,\nu}\,\lambda_{\nu+j-1}
   +A_{\nu+j,\nu}\,\lambda_{\nu+j},
& \text{if} -\nu+1\leq j\leq 0,\\[6pt]
\begin{aligned}
\displaystyle&
\sum_{k=\nu-1}^{\nu+j-2}
(-1)^{\nu+j-1-k}
\frac{b_{k-\nu+1}\cdots b_{j-1}\,C_{k,\nu}\,\lambda_{k}}
     {a_{k-\nu+2}\cdots a_{j}}\\
&\quad +\,C_{\nu+j-1,\nu}\,\lambda_{\nu+j-1}
      +A_{\nu+j,\nu}\,\lambda_{\nu+j},
\end{aligned}
& \text{if } j\ge 1.
\end{cases}
\]
With the above matrix powers in our hand, we now have the following characterizations.
 \begin{theorem}\label{hyper5}
      The following hold for the bilateral weighted backward shift $B_{w}$ acting on  $\ell^{p}_{a,b}(\Omega_{r,R})$, $1\leq p<\infty,$ assuming that the conditions \eqref{rr} and \eqref{tt} of Theorem \ref{decomposition1} are satisfied:
     \begin{enumerate}
      \item[(i)] $B_{w}$ is hypercyclic on $\ell^{p}_{a,b}(\Omega_{r,R})$ if and only if
         \begin{equation}\label{hyperc}
              \limsup_{\nu\rightarrow +\infty}\left\lvert w_{n+1}\cdots w_{n+\nu}{a_{n+\nu}}   
 \right\rvert=\infty= \limsup_{\nu\rightarrow +\infty} \left\lvert \frac{a_{n-\nu}}{w_{n}\cdots w_{n-\nu+1}} \right\rvert,~~ \forall~~ n\in \mathbb{N}.
        \end{equation}
         \item[(ii)] $B_{w}$ is topologically mixing on $\ell^{p}_{a,b}(\Omega_{r,R})$ if and only if
         \begin{equation}\label{mixing}
         \lim_{\nu\rightarrow +\infty}\left\lvert w_{n+1}\cdots w_{n+\nu}{a_{n+\nu}}   
 \right\rvert=\infty= \lim_{\nu\rightarrow +\infty} \left\lvert \frac{a_{n-\nu}}{w_{n}\cdots w_{n-\nu+1}}\right\rvert,~~ \forall~~ n\in \mathbb{N}.
         \end{equation}

         \item[(iii)] $B_{w}$ is supercyclic on $\ell^{p}_{a,b}(\Omega_{r,R})$ if and only if
         \begin{equation}\label{superc}
         \limsup_{\nu\rightarrow +\infty} \left\lvert 
         \frac{w_{n+1}\cdots w_{n+\nu}}{w_{n}\cdots w_{n-\nu+1}}
          a_{n+\nu}a_{n-\nu}\right\rvert=\infty, \hskip .5cm \forall~~ n\in \mathbb{N}.
         \end{equation}
        \end{enumerate}
    \end{theorem}
    \begin{proof}
           (i) Suppose that (\ref{hyperc}) holds, and we will use Gethner-Shapiro criterion to show that $B_{w}$ is hypercyclic. Let $D$ be the span of $\{z^{n}:n\in \mathbb{Z}\}$. Then, $D$ is dense in $\ell^{p}_{a,b}(\Omega_{r,R}),$ as it contains the normalized Schauder basis vectors $a_{-n}z^{-n}$ and $(a_{n}+b_{n}z)z^{n}$. Define $S:D\rightarrow D$ by 
           \begin{center}
           $S(z^{n})=\frac{1}{w_{n+1}}z^{n+1},~~ \forall~~ n\in \mathbb{Z}.$
           \end{center}
           Since $B_{w}S(f)=f,$ $\forall$ $f\in D,$ we just need to show that, for some increasing sequence $(\nu_{k})\subset \mathbb{N},$ both the sequences $B_{w}^{\nu_{k}}(z^{n})$ and $S^{\nu_{k}}(z^{n})$ tend to 0 as $k\to +\infty,$ for all $n\in \mathbb{N}.$ Note that
       \[
       B^{\nu}_{w}(z^{n}) =w_{n}w_{n-1}\cdots w_{n-\nu+1}z^{n-\nu}, 
       \]
       and 
       \[
       S^{\nu}(z^{n}) =\frac{1}{w_{n+1}\cdots w_{n+\nu}}z^{n+\nu}.
       \]
       Combining these expressions with Proposition \ref{estimate} and the assumption (\ref{hyperc}), we can see that $B_{w}$ satisfies the Gethner-Shapiro criterion. Hence, $B_{w}$ is hypercyclic on $\ell^{p}_{a,b}(\Omega_{r,R}).$
       
      For the converse, let $B_w$ be hypercyclic. Then, by a result of Bonet (cf. \cite{Bonet1}), we have, for fixed $n\in \mathbb{N}$,
\[
\sup_{\nu\geq 1}\|B_w^{*\nu}(k_n)\|=\infty,
\]
which, by Propositions \ref{k_{n}} and \ref{act-like}, together with the hypothesis  
\(\displaystyle \limsup_{\nu\to +\infty}|b_\nu/a_{\nu+1}|<1,\)  
implies that  
$
\sup_{\nu\geq 1}\,\bigl|w_{n+1}\cdots w_{n+\nu}\,a_{n+\nu}\bigr|=\infty.
$
It remains to show that  
\[
\limsup_{\nu\to +\infty}\left|\frac{a_{n-\nu}}{w_n\cdots w_{n-\nu+1}}\right|=\infty,
\qquad \forall n\in\mathbb{N},
\]
for which we recall equation \eqref{cc}. Let $u=(\lambda_j)_{j\in\mathbb{Z}}\in\ell^p(\mathbb{Z})$ be a hypercyclic vector for the matrix operator $[B_w]$. Fix $n\geq 1$ and $0<\delta<1$. Then, we can find an integer $\nu>2 n$ such that  $\|[B_w^\nu]u - e_n\| < \delta$, i.e.
\[
\left\lVert \sum_{j=-\infty}^{+\infty} \alpha_{j,\nu} ~ e_{j}- e_n\right\rVert < \delta.
\]
Examining the $(n-\nu)$-th coordinate in the above inequality gives
\[
|C_{n-1,\nu}\lambda_{n-1} + A_{n,\nu}\lambda_n| < \delta
\quad\Rightarrow\quad
\left|\frac{w_n\cdots w_{n-\nu+1}}{a_{n-\nu}}\bigl(\lambda_{n-1}b_{n-1}+\lambda_n a_n\bigr)\right| < \delta.
\]
Since $\lambda_{n-1}b_{n-1}+\lambda_n a_n\neq 0$, and $\delta>0$ is arbitrary, this yields the second part of the converse implication.\vspace{3mm}

       (ii) To conclude the mixing property of $B_w$, proceed as in the proof of (i) by taking $\nu_k=k$ for all $k\geq 1$, in the Gethner-Shapiro criterion. \vspace{3mm}

        (iii) We now prove the supercyclicity. Suppose that (\ref{superc}) holds. In the supercyclicity criterion, let $D$ and $S$ be defined as above. It suffices to verify that, there exists an increasing sequence $(\nu_{k})\subset \mathbb{N},$ such that
        $ \lVert B^{\nu_{k}}_{w}(z^{n}) \rVert \lVert S^{\nu_{k}}(z^{n}) \rVert$ tends to 0 as $k\to +\infty,$ for all $n\in \mathbb{N}.$ Since we have
        \[
         \lVert B^{\nu}_{w}(z^{n}) \rVert_{\ell^{p}_{a,b}(\Omega_{r,R})} \lVert S^{\nu}(z^{n}) \rVert_{\ell^{p}_{a,b}(\Omega_{r,R})}=\lVert   w_{n}w_{n-1}\cdots w_{n-\nu+1}z^{n-\nu} \rVert\left\lVert \frac{z^{n
         +\nu}}{w_{n+1}w_{n+2}\cdots w_{n+\nu}}\right\rVert,
        \]
        Proposition~\ref{estimate}, together with assumption~(\ref{superc}),  implies that \(B_w\) satisfies the supercyclicity criterion.\\
               For the converse, assume that $B_{w}$ is supercyclic. Consequently, the matrix transformation $[B_w]$ is supercyclic on $\ell^p(\mathbb{Z})$. Fix $n\in \mathbb{N}$, $0<\delta<1$, and a supercyclic vector 
$u=(\lambda_j)_{j\in\mathbb{Z}}\in \ell^{p}(\mathbb{Z})$ for the matrix transformation $[B_w]$. Choose $\lambda \in \mathbb{C}\setminus\{0\}$ and $\nu>2n$ such that  
\begin{equation}
    \|\lambda[B_w^\nu]u - e_n\| < \delta,
\end{equation}
i.e., 
\begin{equation}\label{nece-sup-ine}
 ~\Big\lVert \lambda\sum_{j=-\infty}^{+\infty} \alpha_{j,\nu}\, e_{j}- e_{n}\Big\rVert < \delta.
\end{equation}
Inspecting the $(n-\nu)$-th coordinate in \eqref{nece-sup-ine}, we obtain  
\begin{equation}\label{conver-super1}
     \left|\frac{w_n\cdots w_{n-\nu+1}}{a_{\,n-\nu}}\right|
        < \frac{\delta}{|\lambda|\left|\lambda_{n-1} b_{n-1}+\lambda_n a_n\right|}.
\end{equation}
Next, examining the $n$-th coordinate in \eqref{nece-sup-ine}, we get
\[
\left\lvert\lambda\left(\sum_{k=\nu-1}^{\nu+n-2}
(-1)^{\nu+n-1-k}
\frac{b_{k-\nu+1}\cdots b_{n-1}\,C_{k,\nu}\,\lambda_{k}}
     {a_{k-\nu+2}\cdots a_{n}}+\,C_{\nu+n-1,\nu}\,\lambda_{\nu+n-1}
      +A_{\nu+n,\nu}\,\lambda_{\nu+n}\right)-1\right\rvert<\delta,
\]
which implies that
\begin{equation}\label{main-nece-super}
\begin{split}
\frac{1-\delta}{|\lambda|}
&<
\left|
\frac{w_{n+1}\cdots w_{n+\nu}a_{n+\nu}}{a_{n}}
\bigl(\lambda_{n+\nu}+\lambda_{n+\nu-1}\frac{b_{n+\nu-1}}{a_{n+\nu}}\bigr)
\right|\\
&\quad+ \sum_{j=0}^{n-1}
\left|
\frac{b_{j}\cdots b_{n-1}}{a_{j}\cdots a_{n}}\,
w_{j+1}\cdots w_{\nu+j}\,
\bigl(\lambda_{\nu+j}a_{\nu+j}+\lambda_{\nu-1+j}b_{\nu-1+j}\bigr)
\right|.
\end{split}
\end{equation}
For $j=0,\cdots, n-1,$ the equation \eqref{nece-sup-ine} yields  
$|\alpha_{j,\nu}|<\frac{\delta}{|\lambda|}.$
In particular, for $j=0$ we obtain  
\[
\left| w_{1}\cdots w_{\nu}\,(\lambda_{\nu}a_{\nu}+\lambda_{\nu-1}b_{\nu-1}) \right|
   <\frac{\delta\,|a_{0}|}{|\lambda|}.
\]
Consequently, for each $j=0,\cdots, n-1$, one obtains that 
\[
\left|
\frac{b_{j}\cdots b_{n-1}}{a_{j}\cdots a_{n}}\,
w_{j+1}\cdots w_{\nu+j}\,
\bigl(\lambda_{\nu+j}a_{\nu+j}
      +\lambda_{\nu-1+j}b_{\nu-1+j}\bigr)
\right|
< M_{j}\,\frac{\delta}{|\lambda|},
\]
for some constant $M_{j}>0$. Therefore, for a suitable constant $M>0$, the equation \eqref{main-nece-super} implies  
\begin{equation}\label{new-super}
    \frac{1}{\left|w_{n+1}\cdots w_{n+\nu} a_{n+\nu}\right|}
   <
   \frac{|\lambda|\left|\lambda_{n+\nu-1}\tfrac{b_{n+\nu-1}}{a_{n+\nu}}+\lambda_{n+\nu}\right|}
        {|a_{n}|\,\bigl(1-\delta(1+M)\bigr)}.
\end{equation}
Combining \eqref{conver-super1} with \eqref{new-super}, we deduce that  
\[
 \left\lvert 
         \frac{w_{n}\cdots w_{n-\nu+1}}{w_{n+1}\cdots w_{n+\nu}}\frac{1}{
          a_{n+\nu}a_{n-\nu}}
 \right\rvert
\leq 
\left|\lambda_{n+\nu-1}\tfrac{b_{n+\nu-1}}{a_{n+\nu}}+\lambda_{n+\nu}\right|
\frac{\delta}{|a_{n}|\bigl(1-\delta(1+M)\bigr)\left|\lambda_{n-1} b_{n-1}+\lambda_n a_n\right|}.
\]
Since $(\lambda_j)_{j\in\mathbb{Z}}\in \ell^{p}(\mathbb{Z})$ and  
$|b_{n+\nu-1}/a_{n+\nu}|<r$ for some $r\in(0,1)$, we have  
\[
\left|\lambda_{n+\nu-1}\frac{b_{n+\nu-1}}{a_{n+\nu}}+\lambda_{n+\nu}\right|
   \leq 
|\lambda_{n+\nu-1}|\, r + |\lambda_{n+\nu}|.
\]
Assume that we have an additional requirement $\delta<\tfrac{1}{1+M}$. Since $a_{n}\neq 0$ and   
$\lambda_{n-1}b_{n-1}+\lambda_n a_n\neq 0$, we can conclude that  
\[
\liminf_{\nu\to\infty}  
\left\lvert 
\frac{w_{n}\cdots w_{n-\nu+1}}{w_{n+1}\cdots w_{n+\nu}}\frac{1}{
          a_{n+\nu}a_{n-\nu}}
\right\rvert
= 0.
\]
This completes the proof.

   \end{proof}
   
     We illustrate the above results with an example. For the same, let $\Omega_{\frac{1}{2},1}$ denote the annulus $\frac{1}{2}<|z|<1$. Recall the  Bergman space $A^{2}(\Omega_{\frac{1}{2},1})$, consisting of all square integrable holomorphic functions on $\Omega_{\frac{1}{2},1}$ i.e.,
 \[
A^{2}(\Omega_{\frac{1}{2},1})=\{f:~\Omega_{\frac{1}{2},1}\to \mathbb{C}~\text{holomorphic and} ~\lVert f \rVert^{2}:=\frac{1}{2\pi}\int_{\Omega_{\frac{1}{2},1}}|f(z)|^{2}~dA(z)<\infty\},
 \]
 where $dA(z)=dxdy$ is the area measure. (We refer to Hedenmalm et al. \cite{Zhu} and Shields \cite{Shields} for details on Bergman spaces.) For $f(z)=\sum_{n=-\infty}^{+\infty}\widehat{f}(n)z^{n}\in A^{2}(\Omega_{\frac{1}{2},1}),$ the norm is given by 
 \begin{center}
 $\lVert f\rVert^{2}_{A^{2}(\Omega_{\frac{1}{2},1})}=\sum_{n=-\infty}^{+\infty}|\widehat{f}(n)|^{2}\gamma_{n},$
 \end{center}
 and the set $\left\{ \gamma_{n}^{-\frac{1}{2}}z^{n}\right\}_{n\in\mathbb{Z}}$ is an orthonormal basis of $A^{2}(\Omega_{\frac{1}{2},1}),$ where 
\[
\gamma_{n}=\frac{1-2^{-2(n+1)}}{2(n+1)},\quad \forall~~ n\in\mathbb{Z}\backslash\{-1\}~~\text{and}~~\quad \gamma_{-1}=\log 2.
\]
 
  Let $\ell^{2}_{a,b}(\Omega_{\frac{1}{2},1})$ be the Hilbert space of all analytic functions on the annulus $\Omega_{\frac{1}{2},1},$ having an orthonormal basis of the form  
\[
f_n(z)=
\begin{cases}
   \left( \gamma_{n}^{-\frac{1}{2}}+z\right)z^{n}, \hskip.5cm n\geq 0,\\  \gamma_{n}^{-\frac{1}{2}} z^{n},\hskip 1.8cm n\leq -1,
\end{cases}
\]
then, the evaluation functionals are bounded, and $\ell^{2}_{a,b}(\Omega_{\frac{1}{2},1})$ is densely and continuously included in the Bergman space  $A^{2}(\Omega_{\frac{1}{2},1})$. Indeed, if $f\in \ell^{2}_{a,b}(\Omega_{\frac{1}{2},1})$, then there exists some $\{\lambda_n\}_{n\in\mathbb{Z}}\in\ell^2(\mathbb{Z})$ such that $f(z)=\sum_{n=-\infty}^{
+\infty} \lambda_n f_n(z)$ for all $z\in \Omega_{\frac{1}{2},1}$. Rearranging the sum as a Laurent series, we get 
\[
f(z)=\sum_{n=-\infty}^{0}\lambda_{n} \gamma_{n}^{-\frac{1}{2}}z^{n}+\sum_{n=1}^{+\infty} \left(\lambda_{n-1}+\lambda_n \gamma_{n}^{-\frac{1}{2}}\right)z^n,
\]
and so, for some constant $M>0$ we have
\[
 \lVert f\rVert^{2}_{A^{2}(\Omega_{\frac{1}{2},1})}=\sum_{n=-\infty}^{+\infty}|\widehat{f}(n)|^{2}\gamma_{n}\leq M^2\sum_{n=-\infty}^{+\infty}|\lambda_n|^2<\infty,
\]
since $\{\lambda_n\}\in\ell^2(\mathbb{Z})$ and where
\[
\widehat{f}(n) =
\begin{cases}
\lambda_{n}\gamma_{n}^{-\frac{1}{2}}, & n\le 0,\\[0.7em]
\lambda_{n-1}+\lambda_{n}\gamma_{n}^{-\frac{1}{2}}, & n\ge 1.
\end{cases}
\]
Hence, $\lVert f\rVert^{2}_{A^{2}(\Omega_{\frac{1}{2},1})}\leq M \|f\|_{\ell^{2}_{a,b}(\Omega_{\frac{1}{2},1})}$, for all $f\in \ell^{2}_{a,b}(\Omega_{\frac{1}{2},1})$.
 \begin{proposition}
Let $\ell^{2}_{a,b}(\Omega_{\frac{1}{2},1})$ be the Hilbert space defined as above. Then the inclusion from $\ell^{2}_{a,b}(\Omega_{\frac{1}{2},1})$ into $A^{2}(\Omega_{\frac{1}{2},1})$ is continuous and has dense range. Moreover, if $w_{n}=4$ for all $n\geq 0$, and $w_{n}=\tfrac{1}{4}$ for all $n\leq -1$, then the shift $B_{w}$ on $\ell^{2}_{a,b}(\Omega_{\frac{1}{2},1})$ is mixing, and similar to a compact perturbation of a classical weighted bilateral backward shift.
\end{proposition}
\begin{proof}
    Recalling from the proof of Theorem \ref{decomposition1}, we have that $B_{w}$ on $\ell^{2}_{a,b}(\Omega_{\frac{1}{2},1})$ is 
    \[
       T_{-1}+D+\sum_{i=1}^{+\infty}T_i.
     \]
     Also, for the given normalized Schauder basis $\{f_n\}_{n\in \mathbb{Z}}$ and $w_{n}=4,$ for all $n\geq 0,$ we have
     \[
     \lim_{ n \rightarrow +\infty} \left\lvert w_{n+1}\frac{b_{n}}{a_{n}}-w_{n}\frac{b_{n-1}}{a_{n-1}}\right\rvert=\lim_{ n \rightarrow +\infty} 4\left\lvert \sqrt{\frac{1-2^{-2(n+1)}}{2(n+1)}} -\sqrt{\frac{1-2^{-2n}}{2n}}\right\rvert=0.
     \]
    Therefore, according to Theorem \ref{prop}, the operators $D$ and $T_i~~(i\geq 1)$ are compact on $\ell^2(\mathbb{Z})$ as the entries in each of $D$ and $T_i$ converge to $0$. Also, the operator $T_{-1}$ is the Bergman weighted bilateral backward shift on $\ell^2(\mathbb{Z})$ with weights 
      \[
\alpha_{n} =
\begin{cases}
4\sqrt{\frac{\gamma_{n}}{\gamma_{n+1}}}, & n\ge -1,\\
\frac{1}{4}\sqrt{\frac{\gamma_{n}}{\gamma_{n+1}}}, & n\le-2.
\end{cases}
\]
Again, note that $B_{w}$ is mixing on $\ell^{2}_{a,b}(\Omega_{\frac{1}{2},1})$ by Theorem \ref{hyper5} since the sequences
\[
\Big| a_{-n}\prod_{k=0}^{n-1} w_{-k}^{-1} \Big|
= 
4^{n}\sqrt{\frac{2(-n+1)}{1-2^{-2(-n+1)}}},~\text{and}~
\Big| a_{n}\prod_{k=1}^{n} w_{k} \Big|
= 
4^{n}\sqrt{\frac{2(n+1)}{1-2^{-2(n+1)}}},
\]
diverge to $\infty$ as $n\rightarrow \infty$.

\end{proof}

The next result is on the chaos of $B_w$ on $\ell^p_{a,b}(\Omega_{r,R})$, which is an application of the following well known criterion.

\begin{theorem}\emph{\textsf{(Chaoticity Criterion \cite{Bonilla-Erdmann1})}} \label{chaos}
Let $X$ be a separable Banach space, $D$ be a dense set in $X$, and let $T$ be a bounded operator on $X$. If there exists a map $S:D\rightarrow D$ such that 
\[
\sum_{n\geq 0} T^nx~ and~ \sum_{n\geq 0} S^nx
\]
are unconditionally convergent, and $TSx=x$,
for each $x\in D$, then the operator $T$ is chaotic and mixing on $X$. 
\end{theorem}

We also require a result on the unconditional convergence of a series in Banach spaces.
\begin{proposition}\emph{\textsf{(Orlicz–Pettis theorem \cite{Diestel}, p. 24)}}\label{weakly cauchy}
The following are equivalent for a series $\sum_{n}x_{n}$ in a Banach space $X$.
\begin{itemize}
    \item[\textnormal{(i)}] $\sum_{n}x_{n}$ is unconditionally convergent. 
   \item[\textnormal{(ii)}] $\sum_n x_n$ is weakly unconditionally Cauchy, i.e. $\sum_n |x^*(x_n)|<\infty$ for every $x^*\in X^*$.
\end{itemize}
\end{proposition}

We would like to mention that similar results on the chaos for the adjoint of a unilateral weighted forward shift were obtained in \cite{Das-Mundayadan1}. The method of proofs of the following results are essentially similar to the unilateral analogue given in \cite{Das-Mundayadan1}.

    \begin{theorem}\label{chaos1}
          Assume that $a=\{a_n\}_{n=-\infty}^{+\infty}$, $b=\{b_n\}_{n=0}^{+\infty}$, and a weight $w=\{w_n\}_{n=-\infty}^{+\infty}$ satisfy the conditions \eqref{rr} and \eqref{tt} of Theorem \ref{decomposition1}. Consider the following statements for the bilateral weighted backward shift $B_{w}$ on  $\ell^{p}_{a,b}(\Omega_{r,R}),$ $1< p<\infty$ :
          \begin{enumerate}
              \item[(i)] $ B_{w}$ is chaotic on $\ell^{p}_{a,b}(\Omega_{r,R})$.
\item[(ii)] $B_{w}$ has a non-trivial periodic vector.
\item[(iii)]
 $\displaystyle
\sum_{n=1}^{+\infty}\left\lvert  \frac{w_{0}\cdots w_{-n+1}}{a_{-n}}\right\rvert^{p} < \infty ~~~~\text{and}~~~~\sum_{n=1}^{+\infty} \left\lvert\frac{1}{w_{1}\cdots w_{n}a_{n}}\right\rvert^{p}<\infty. 
 $
          \end{enumerate} 
            Then, $\textnormal{(i)} \Rightarrow \textnormal{(ii)}\Rightarrow \textnormal{(iii)}.$ \\ Additionally, if $\frac{1}{p}+\frac{1}{q}=1,$ and 
            \[
            \sum_{n\geq 1}\left(\sum_{j\geq 1}\left\lvert \Mprod_{k=1}^{j}\frac{b_{n+k-1}}{ a_{n+k}}\right\rvert^{p}\right)^{\frac{q}{p}}<\infty,
            \]
            then, we get the characterizations: $\textnormal{(i)} \Leftrightarrow \textnormal{(ii)} \Leftrightarrow \textnormal{(iii)}.$
    \end{theorem}

    \begin{proof}
It is trivial that $\textnormal{(i)} \Rightarrow \textnormal{(ii)} $.
   
To prove $\textnormal{(ii)} \Rightarrow \textnormal{(iii)}$, let $f(z)=\sum_{n=-\infty}^{+\infty}\lambda_{n}f_{n}(z)$ be a non-zero periodic vector for $B_{w}$ on  $\ell^{p}_{a,b}(\Omega_{r,R}).$ Writing this as a Laurent series, we have $f(z)=\sum_{n=-\infty}^{+\infty} \widehat{f}(n)z^{n}$, where 
 $$\widehat{f}(n)=
\begin{cases}
   \lambda_{n}a_n+\lambda_{n-1}b_{n-1}, \hskip.5cm n\geq 1,\\ \lambda_{n}a_n,\hskip 2.6cm n\leq 0.
\end{cases}$$
 Let us choose $m \in \mathbb{N}$ such that $B_{w}^{m}f(z)=f(z),$ for all $z$ in the annulus $\Omega_{r,R}$. Since $B_{w}^{km}f(z)=f(z)$ for all $k\geq 1$, it follows that 
 \begin{eqnarray*}
\sum_{n=-\infty}^{+\infty} w_{n}\cdots w_{n-km+1}\widehat{f}(n)z^{n-km}=\sum_{n=-\infty}^{+\infty} \widehat{f}(n)z^{n},
\end{eqnarray*}
for all $z\in \Omega_{r,R}$. As $m$ be the period of $B_{w},$ then there is some $\lvert t\rvert\leq m$ such that $\lambda_{j}\neq 0,$ for all $j\geq \lvert t\rvert.$  Therefore, using the coefficients of like powers for comparison, we get
  $w_{j+1}\cdots w_{j+km}\widehat{f}(j+km)=\widehat{f}(j)$, $\forall~~ k \geq 1,$  
where $t\leq j \leq m-1$ and $t\in\mathbb{N}$. For $j=t,$ we have $w_{t+1}\cdots w_{t+km}(\lambda_{t+km}a_{t+km}+\lambda_{t+km-1}b_{t+km-1})=\lambda_{t}a_{t}$, for $ k \geq 1.$ Writing $s_{n,j}:=w_{j+1}\cdots w_{j+n}$ for convenience, we see that for some constant $C_1>0$
\[
\lvert\lambda_{t}a_{t}\rvert^{p}\sum_{k=1}^{+\infty}\left\lvert\frac{1}{s_{km,t}a_{t+km}}\right\rvert^{p}\leq C_{1}r^{p}\big(\sum_{k=1}^{+\infty}\left\lvert\lambda_{t+km}\right\rvert^{p}+\sum_{k=1}^{+\infty}\left\lvert\lambda_{t+km-1}\right\rvert^{p}\big).
\]
 Here, $r:=\sup_{n\geq 0}\left\lvert b_{n}/a_{n+1}\right\rvert$. As $\{\lambda_{n}\} \in \ell^{p}(\mathbb{Z})$, we have 
$\sum_{k=1}^{+\infty}\left\lvert\frac{1}{s_{km,t}a_{t+km}}\right\rvert^{p}<\infty.$ Along the same lines, for $t+1\leq j\leq m-1,$ we have
$
w_{j+1}\cdots w_{j+km}(\lambda_{j+km}a_{j+km}+\lambda_{j+km-1}b_{j+km-1})=\lambda_{j}a_{j}+\lambda_{j-1}b_{j-1},~\text{and}~\sum_{k=1}^{+\infty}\left\lvert\frac{1}{s_{km,j}a_{j+km}}\right\rvert^{p}<\infty.
$
From this, we get the convergence of 
$\sum_{n=1}^{+\infty} \left\lvert\frac{1}{w_{1}\cdots w_{n}a_{n}}\right\rvert^{p}<\infty.$ In a similar way, it can be proved that $\sum_{n=1}^{+\infty}\left\lvert  \frac{w_{0}\cdots w_{-n+1}}{a_{-n}}\right\rvert^{p} $ is convergent. Indeed, for $t\leq j \leq m-1,$ $t\in\mathbb{N}$ and again using the coefficients of like powers for comparison, we obtain $w_{-j}\cdots w_{-j-km+1}\widehat{f}(-j)=\widehat{f}(-j-km)$, $\forall~~ k \geq 1,$ implying that 
$w_{-j}\cdots w_{-j-km+1}\lambda_{-j}a_{-j}=\lambda_{-j-km}a_{-j-km}$, $\forall~~ k \geq 1.$ Thus
\begin{center}
$|\lambda_{-j}a_{-j}|^{p}\displaystyle\sum_{k=1}^{+\infty} \left\lvert\frac{w_{-j}\cdots w_{-j-km+1}}{a_{-j-km}}\right\rvert^{p}=\sum_{k=1}^{+\infty} \left\lvert\lambda_{-j-km}\right\rvert^{p}.$
\end{center}
 As $\{\lambda_{n}\} \in \ell^{p}(\mathbb{Z})$, we have $\sum_{k=1}^{+\infty} \left\lvert\frac{w_{-j}\cdots w_{-j-km+1}}{a_{-j-km}}\right\rvert^{p}<\infty.$ Consequently, the series\\ $\sum_{n=1}^{+\infty}\left\lvert  \frac{w_{0}\cdots w_{-n+1}}{a_{-n}}\right\rvert^{p}$ is convergent and thus the implication $\textnormal{(ii)} \Rightarrow \textnormal{(iii)}$ follows.

 We now prove that  $\textnormal{(iii)}$ implies $\textnormal{(i)}$ by applying the chaoticity criterion. Take $D$ and $S$ to be the same set and map as in the previous theorem. Observing that $B_{w}S(f)=f,$ $\forall$ $f\in D,$ we need only to prove that $\sum_{n=1}^{+\infty}w_{0}\cdots w_{-n+1}z^{-n}$ and  $ \sum_{n=1}^{+\infty} \frac{1}{w_{1}\cdots w_{n}}z^{n}$ are unconditionally convergent in $\ell^{p}_{a,b}(\Omega_{r,R}).$ By the Orlicz-Pettis theorem it is sufficient to prove that the series are weakly unconditionally Cauchy. Now using Proposition \ref{duality}, we have if $L\in (\ell^{p}_{a,b}(\Omega_{r,R}))^{*},$ then there exists $\{y_{n}\}_{n=-\infty}^{+\infty}\in \ell^{q}(\mathbb{Z})$ such that $L\left(\sum_{n=-\infty}^{+\infty}\lambda_{n}f_{n}(z)\right)=\sum_{n=-\infty}^{+\infty}\lambda_{n}y_{n}.$ So, for all $n\geq 1$
\[
L(z^{n})=\frac{1}{a_{n}}\sum_{j=0}^{+\infty}\lambda_{n,j}L(f_{n+j})=\frac{1}{a_{n}}\sum_{j=0}^{+\infty}\lambda_{n,j}y_{n+j},
\]
where $\lambda_{n,0}=1$ and $ \lambda_{n,j}=(-1)^{j}\frac{b_{n}\cdot\cdot \cdot b_{n+j-1}}{a_{n+1}\cdot\cdot \cdot a_{n+j}},$ for all $n\geq 1,~~~ j\geq 1.$
Hence, for some constant $C>0,$ we obtain
\begin{eqnarray*}
    \sum_{n=1}^{+\infty}\left\lvert\frac{1}{w_1\cdots w_n}L(z^n)\right\rvert&\leq& \sum_{n=1}^{+\infty}\frac{1}{|w_{1}\cdots w_{n}a_{n}|}\left\lvert\sum_{j=0}^{+\infty}\lambda_{n,j}y_{n+j}\right\rvert\\
    &\leq&C\left(\sum_{n=1}^{+\infty}\frac{1}{|w_{1}\cdots w_{n}a_{n}|^{p}}\right)^{\frac{1}{p}}\left(1+\left(\sum_{n=1}^{+\infty}\left(\sum_{j=1}^{+\infty}|\lambda_{n,j}|^p\right)^{\frac{q}{p}}\right)^{\frac{1}{q}}\right),
\end{eqnarray*}
which is finite as $\left\{\frac{1}{w_{1}\cdots w_{n}a_n}\right\}_{n=1}^{+\infty}\in \ell^{p}(\mathbb{N})$ and  $\sum_{n\geq 1}\left(\sum_{j\geq 1}\left\lvert \Mprod_{k=1}^{j}\frac{b_{n+k-1}}{ a_{n+k}}\right\rvert^{p}\right)^{\frac{q}{p}}<\infty.$ It remains to show that $\sum_{n=1}^{+\infty} |w_{0}\cdots w_{-n+1}L(z^{-n})|$ is  convergent. The proof is similar, and hence we do not show the computations. We can conclude that $B_{w}$ satisfies the chaoticity criterion, and (i) follows.
    \end{proof}

The proofs of the hypercyclicity, mixing, supercyclicity, and chaos for $B_w$ show that, one has more general results on the dynamics of $B_w$ when $B_w$ is defined on a Banach space of analytic functions, stated as follows. The proofs are omitted.

\begin{theorem}\label{gen1}
    Let $\mathcal{E}$ be a separable Banach space of analytic functions on an annulus $\Omega$ concentric at $0$, having bounded evaluation functionals, and let $\mathcal{E}$ contain all Laurent polynomials. If $B_w$ is bounded on $\mathcal{E}$, then the following hold.
    \begin{enumerate}
         \item[(i)] $B_{w}$ is hypercyclic on $\mathcal{E}$ if\\  
         $\liminf_{\nu\rightarrow +\infty}~ \lvert w_{n}\cdots w_{n-\nu+1}\rvert\lVert z^{n-\nu}\rVert=0=\liminf_{\nu\rightarrow +\infty}~  \frac{\lVert z^{n+\nu }\rVert}{\lvert w_{n+1}\cdots w_{n+\nu}\rvert},$ $~~ \forall~~ n\in \mathbb{N}.$
        \item[(ii)] $B_w$ is mixing on $\mathcal{E}$ if\\
         $\lim_{\nu\rightarrow +\infty}~ \lvert w_{n}\cdots w_{n-\nu+1}\rvert\lVert z^{n-\nu}\rVert =0= \lim_{\nu\rightarrow +\infty}~   \frac{\lVert z^{n+\nu }\rVert}{\lvert w_{n+1}\cdots w_{n+\nu}\rvert},~~ \forall~~ n\in \mathbb{N}.$
         \item[(iii)] $B_{w}$ is  supercyclic on $\mathcal{E}$ if\\  
         $
         \liminf_{\nu\rightarrow +\infty} \left|\frac{w_{n}\cdots w_{n-\nu+1} }{w_{n+1}\cdots w_{n+\nu}}\right|\|z^{n-\nu}\|\|z^{n+\nu}\|=0,~~ \forall~~ n\in \mathbb{N}.
         $
         \item[(iv)]  $B_w$ is chaotic on $\mathcal{E}$ if\\
         $\sum_{\nu\geq 1} \Big|\frac{1}{w_{1}\cdots w_{\nu}} L(z^{\nu})\Big|<\infty
       \hskip .4cm \text{and}\hskip .4cm  \sum_{\nu\geq 1} \Big |w_{0}\cdots w_{-\nu+1} L(z^{-\nu})\Big |<\infty,~~\forall~ L\in \mathcal{E^{*}}.
         $
     \end{enumerate}
\end{theorem}

Next, we state a result on the existence of hypercyclic subspaces for $B_w$. It is well known that an operator $T$ on a complex Banach space, satisfying the hypercyclicity criterion has a hypercyclic subspace if and only if 
\[
\sigma_e(T)\cap S^1\neq \phi,
\]
$S^1$ being the unit circle in $\mathbb{C}$. Every hypercyclic weighted shift on $\ell^p(\mathbb{Z})$ has a hypercyclic subspace. This is not true for hypercyclic unilateral shifts, and a characterization along with other results can be found in \cite{Bayart-Matheron}, \cite{Leon}, \cite{Menet}, and \cite{Montes}.

\begin{theorem}
    Suppose $B_w$ satisfies the conditions of Theorem \ref{prop}. If $B_w$ is hypercyclic, then it has hypercyclic subspaces.
\end{theorem}

We provide explicit examples to illustrate our results on the linear dynamics of $B_w$. Particularly, we give a class of chaotic operators that are compact perturbations of some weighted  shifts on $\ell^p(\mathbb{Z})$, see Remark \ref{compact-remark} below.

 \begin{example}
 Let
 \[
a_n =
\begin{cases}
1, & n \ge 0,\\[4pt]
n\,2^{\,n-2}, & n \le -1,
\end{cases}
\qquad
b_n =
\begin{cases}
1, & n = 0,\\[4pt]
\frac{1}{n}, & n \ge 1,
\end{cases}
\qquad
w_n =
\begin{cases}
2, & n \ge 0,\\[4pt]
2^{n}, & n \le -1.
\end{cases}
\]
 Therefore, we have 
 \[
 r=\limsup_{n\rightarrow +\infty}\left(\frac{n}{2^{n+2}}\right)^{1/n}=\frac{1}{2}\qquad \text{and}\qquad\frac{1}{R}=\limsup_{n\rightarrow +\infty}\left(1+\frac{1}{n}\right)^{1/n}=1.
 \]
 Hence, the annulus $\Omega_{r,R}$ of $\ell^p_{a,b}(\Omega_{r,R})$ is $r<|z|<R$, where $r=\frac{1}{2}$ and $R=1.$ In view of Theorem \ref{decomposition1}, $B_{w}$ is bounded on $\ell^p_{a,b}(\Omega_{\frac{1}{2},1}).$ Additionally, $B_w$ satisfies the conditions of Theorems \ref{hyper5} and \ref{chaos1}, and so $B_{w}$ is mixing and chaotic on $\ell^p_{a,b}(\Omega_{\frac{1}{2},1}).$  Indeed,
    \[
    \lim_{n\rightarrow +\infty}\left\lvert a_{-n}\Mprod_{k=0}^{n-1} w_{-k}^{-1}
 \right\rvert=\lim_{n\rightarrow +\infty}\frac{n\prod_{k=1}^{n-1}2^{k}}{2^{n+3}}=\infty ~~~~\text{and}~~~~ \lim_{n\rightarrow +\infty} \left\lvert a_{n}\Mprod_{k=1}^{n} w_{k} \right\rvert=\lim_{n\to+\infty} 2^{n}=\infty.
    \]
    \end{example}
    \begin{example}
    In the previous example, if we let 
    \[
w_n =
\begin{cases}
1, & n \ge 0,\\[4pt]
2^{n}, & n \le -1,
\end{cases}
\]
 and $\{a_{n}\}, \{b_{n}\}$ being the same as mentioned above, then we see that, by Theorem \ref{hyper5}, $B_{w}$ is supercyclic. Again, by Theorem \ref{hyper5}, we know that if $B_{w}$ is hypercyclic on $\ell^p_{a,b}(\Omega_{\frac{1}{2},1}),$ then  $\sup_{n\geq 1}~(|w_{1}\cdots w_{n}a_{n}|)=\infty$.  However, in this case  $B_{w}$ is not hypercyclic on $\ell^p_{a,b}(\Omega_{\frac{1}{2},1})$ as
   $
    \sup_{n\geq 1} \left\lvert a_{n}\Mprod_{k=1}^{n}w_{k} \right\rvert=1.
    $
\end{example}

\begin{remark}\label{compact-remark}
It should be noted that the operators $B_w$ considered in this section satisfy
\[
\lim_{ n \rightarrow +\infty} \left\lvert w_{n+1}\frac{b_{n}}{a_{n}}-w_{n}\frac{a_n}{a_{n-1}}\frac{b_{n-1}}{a_{n}}\right\rvert=0.
\]
Therefore, by the Theorem \ref{prop}, we see that $B_{w}$ on $\ell^p_{a,b}(\Omega_{r,R})$ is similar to a compact perturbation of a bilateral weighted backward shift on $\ell^{p}(\mathbb{Z}).$ Consequently, we have a class of chaotic operators that are compact perturbation of some classical weighted shifts.
\end{remark}

\subsection{Orbital limit points versus hypercyclicity}

This section is devoted to the study of the zero-one law of hypercyclicity for $B_w$. In our context, this law fails, as seen in the next proposition. Compare our results with the work of Chan and Seceleanu \cite{Kit1}, and Abakumov and Abbar \cite{Abakumov-Abbar}, who respectively proved that such a law holds for the hypercyclicity and supercyclicity of the classical bilateral weighted shifts, (see Section $1$).

\begin{proposition}\label{Not dense}
Let\[
a_n =
\begin{cases}
1, & n \ge 0,\\[4pt]
2^{n}, & n < 0,
\end{cases}
\qquad
b_n =\frac{1}{2}, ~~ n \ge 0,
\qquad~\text{and}\quad
w_n =2,  ~~ n \in \mathbb{Z}.
\]
 Then, the following hold when $1\leq p<\infty$.
 \begin{itemize}
 \item[(i)] $B_w$ is a bounded operator on $\ell^p_{a,b}(\Omega_{\frac{1}{2},1})$ which is actually the multiplication operator $M_{\varphi}$ on $\ell^p_{a,b}(\Omega_{\frac{1}{2},1})$, where $\varphi(z)=\frac{2}{z}$, is defined on the annulus $\Omega_{\frac{1}{2},1}$. Also, $B_w$ similar to a compact perturbation of a bilateral weighted backward shift on $\ell^{p}(\mathbb{Z}).$ 
 \item[(ii)] There exists an infinite dimensional subspace $Y$ of $\ell^p_{a,b}(\Omega_{\frac{1}{2},1})$ such that Orb$(B_{w},f)$ has a nonzero limit point in $\ell^p_{a,b}(\Omega_{\frac{1}{2},1})$, for each non-zero $f\in Y$.
 \item[(iii)] There exists a dense subspace $Z$ of $\ell^p_{a,b}(\Omega_{\frac{1}{2},1})$ such that $\lim_{n\rightarrow \infty}\|B_w^n(g)\|=\infty$, for non-zero $g\in Z$.
 \item[(iv)] $B_w$ is not even supercyclic.
 \end{itemize}
\end{proposition}

\begin{proof}
For the given sequences $a=(a_{n})_{n\in\mathbb{Z}}$ and $b=(b_{n})_{n\in\mathbb{N}_{0}}$, we obtain
\[
r=\limsup_{n\to +\infty}\left(\frac{1}{2^{n}}\right)^{1/n}=\frac{1}{2},
\qquad 
\frac{1}{R}=\limsup_{n\to +\infty}\left(1+\frac{1}{2}\right)^{1/n}=1.
\]
Hence, the annulus $\Omega_{\frac{1}{2},1}$ of $\ell^p_{a,b}(\Omega_{\frac{1}{2},1})$ is $1/2<|z|<1$. By Theorem~\ref{decomposition1}, the operator $B_{w}$ is bounded on $\ell^p_{a,b}(\Omega_{\frac{1}{2},1})$. Note also that, here
\[
\lim_{ n \rightarrow +\infty} \left\lvert w_{n+1}\frac{b_{n}}{a_{n}}-w_{n}\frac{a_n}{a_{n-1}}\frac{b_{n-1}}{a_{n}}\right\rvert=0.
\]
Therefore, by the Theorem \ref{prop}, we see that $B_{w}$ on $\ell^p_{a,b}(\Omega_{\frac{1}{2},1})$ is similar to a compact perturbation of a bilateral weighted backward shift on $\ell^{p}(\mathbb{Z}).$

We now show that the orbital zero-one law for hypercyclicity fails here. For $w_{n}=2,$ $n\in \mathbb{Z}$, we have
\[
A_{n,\nu}=2^{\nu}\frac{a_{n}}{a_{n-\nu}},
\qquad
C_{n,\nu}=
\begin{cases}
0, 
 & \text{if } n\ge \nu,\\[0.4em]

2^{\nu}\,\frac{b_{n}}{a_{n-\nu+1}}, 
 & \text{if } 0\leq n \leq \nu-1,\\[0.4em]

0, & \text{if } n\leq -1.
\end{cases}
\]
Then, for $u=(\lambda_{n})_{n\in \mathbb{Z}}\in \ell^{p}(\mathbb{Z})$, the equation~(\ref{cc}) gives
\[
[B_{w}^{\nu}]u=\sum_{n=-\infty}^{+\infty}\alpha_{n,\nu}\, e_{n},
\]
where $\{e_{n}\}_{n\in \mathbb{Z}}$ is the standard basis in $\ell^p(\mathbb{Z})$, and 
\[
\alpha_{n,\nu}=
\begin{cases}
A_{\nu+n,\nu}\lambda_{\nu+n}, & \text{if } n\leq -\nu,\\[0.5em]
C_{\nu+n-1,\nu}\lambda_{\nu+n-1}+A_{\nu+n,\nu}\lambda_{\nu+n}, & \text{if } -\nu+1\leq n\leq 0,\\[0.5em]
\dfrac{(-1)^{n}b_{0}\cdots b_{n-1}C_{\nu-1,\nu}\lambda_{\nu-1}}{a_{1}\cdots a_{n}}+A_{\nu+n,\nu}\lambda_{\nu+n}, & \text{if } n\geq 1.
\end{cases}
\]
Now, in a special case, let $u=(\lambda_{n})_{n\in \mathbb{Z}}$ be given by
\[
\lambda_{2^{k}}=2^{-2^{k}} \quad \text{for all } k\geq 1, 
\qquad 
\lambda_{j}=0 \ \text{if } j\neq 2^{k}.
\]
Then clearly $u\in \ell^{p}(\mathbb{Z})$, and 
\[
\big\|[B_{w}^{2^{k}}]u - e_{0}\big\|^{p}
   = \sum_{j=k+1}^{\infty}\left(2^{2^{k}}2^{-2^{j}}\right)^{p},
\]
which goes to $0$ as a limit, as $k\to +\infty.$ Thus, $e_{0}$ is a limit point of Orb$([B_{w}],u)$. Now, let 
\[
Y:= \text{span} \{[B_w]^nu:~n\geq 0\}. 
\]
This is the required subspace in (ii).

It is easy to verify that, for the given sequences 
$a = (a_n)_{n \in \mathbb{Z}}$ and 
$b = (b_n)_{n \in \mathbb{N}_0}$, we have 
$z^n \in \ell^p_{a,b}(\Omega_{\frac{1}{2},1})$ for every $n \in \mathbb{Z}$, since
\[
\|z^{n}\|^{p}
=
\sum_{j\geq 0}\frac{1}{2^{jp}}, 
\quad \forall\, n\geq 0,
\qquad \text{and} \qquad
\|z^{n}\|^{p}
=\frac{1}{2^{n}}, 
\quad \forall\, n<0.
\]
We let $Z:=\operatorname{span}\{z^{n} : n \in \mathbb{Z}\},$ which is a dense subspace of $\ell^p_{a,b}(\Omega_{\frac{1}{2},1})$. Furthermore, from \eqref{monomial-expansion}, it follows that for any nonzero $g=\sum_{j=k}^{m}\lambda_j z^{j} \in Z$, where $k,m \in \mathbb{Z}$ with $k<m$, we have
\[
B_{w}^{n}(g)=2^{n}\sum_{j=k}^{m}\lambda_j z^{j-n}=2^{n}\sum_{j=k}^{m}
\frac{\lambda_j}{a_{j-n}} \left(f_{j-n}+\sum_{t=1}^{\infty}
 (-1)^t\frac{b_{j-n}\cdots b_{j-n+t-1}}{a_{j-n+1}\cdots a_{j-n+t}}f_{j-n+t}\right).
\]
Consequently,
\[
\|B_{w}^{n}(g)\|^{p}_{\ell^p_{a,b}(\Omega_{\frac{1}{2},1})}
=
\begin{cases}
\displaystyle 2^{np}\Big(\sum_{j\geq 0}\frac{1}{2^{jp}}\Big)\Big(\sum_{j=k}^{m}|\lambda_{j}|^{p}\Big), 
 & \text{if } j-n\ge 0,\\[0.6em]
2^{2np}\left(\sum_{j=k}^{m}\frac{|\lambda_{j}|^{p}}{2^{jp}}\right), 
 & \text{if } j-n <0,
\end{cases}
\longrightarrow \infty,
\]
as $n \to \infty$. Therefore, there exists a dense subspace $Z$ such that
\[
\lim_{n\to\infty}\|B_w^n(g)\|=\infty,
\quad \text{for all } g \in Z.
\]

Also, observe that  
\[
\limsup_{n\to +\infty} 
\left\lvert 
\frac{w_{1}\cdots w_{n}}{w_{0}\cdots w_{-n+1}}
\, a_{n}a_{-n}
\right\rvert
= 
\limsup_{n\to +\infty} \frac{1}{2^{n}}
= 0.
\]
Hence, by Theorem~\ref{hyper5}, we conclude that \(B_w\) is not even supercyclic on \(\ell^p_{a,b}(\Omega_{\frac{1}{2},1})\). (The non-supercyclicity can also be seen from the following facts: $M_{\varphi}$ cannot be supercyclic since its adjoint has at least two distinct eigenvalues. The point spectrum of the adjoint of a supercyclic operator is either empty or singleton, cf. Bayart and Matheron \cite{Bayart-Matheron}, p. 13).
 \end{proof}

With stronger assumptions, and using the matrix form of $B_w$, we can show the orbital zero-one law holds for $B_w$. Indeed, we have a dichotomy result about the orbits of $B_w$, as seen in the next theorem. For such a dichotomy on the orbits of classical unilateral weighted shifts, we refer to Bonilla et al. \cite{Bonilla-Cardeccia}. The symbol $\mathbb{C}^{\mathbb{Z}}$ will denote the space of all bilateral sequences (as column vectors) of complex sequences, equipped with the standard topology of coordinate wise convergence. 
 
\begin{theorem}\label{T}
    Assume that, for $1\leq p<\infty,$ 
    \begin{equation}\label{yy}
\sup_{n\in \mathbb{Z}}~\left\lvert\frac{ w_{n+1}a_{n+1}}{a_{n}}\right\rvert<\infty, \hspace{.4cm} \limsup_{n\rightarrow +\infty}\left \lvert \frac{b_{n}}{a_{n+1}}\right \rvert<1,~~~~ \text{and} \hspace{.4cm} \sup_{n\geq 1}~\left\lvert\frac{ w_{0}\cdots w_{-n+1}}{a_{-n}}\right\rvert<\infty.
\end{equation}
Then, the following statements are equivalent:
      \begin{itemize}
          \item[(i)] $B_{w}$ is not hypercyclic on $\ell^p_{a,b}(\Omega_{r,R}).$
             \item[(ii)] For each $u \in \ell^p(\mathbb{Z}),$ one has that $[B_{w}^{\nu}]u\to 0$ in $\mathbb{C}^{\mathbb{Z}}$ as $\nu\to +\infty$.
      \end{itemize}
\end{theorem}

\begin{proof}
 For $\nu\geq 1$ and $u=(\lambda_{n})_{n\in \mathbb{Z}}\in \ell^{p}(\mathbb{Z})$, recalling the matrix of $B_{w}^{\nu}$ from (\ref{matrix}) and the action of $[B_{w}^{\nu}]u$ from \eqref{cc}, we have
\begin{equation}\label{ccc}
   [B_{w}^{\nu}]u=\sum_{n=-\infty}^{+\infty} \alpha_{n,\nu} ~ e_{n},
\end{equation}
where $\{e_{n}\}_{n\in \mathbb{Z}}$ is the standard basis in $\ell^p(\mathbb{Z})$, and 
$\alpha_{n,\nu}$ is as given in \eqref{cc}.
Now, assume that $B_{w}$ is not hypercyclic on $\ell^p_{a,b}(\Omega_{r,R})$. By Theorem \ref{hyper5}, we have 
\begin{equation}\label{ddd}
M:=\sup_{\nu\geq 1}\left|w_{n+1}\cdots w_{n+\nu}{a_{n+\nu}}\right|<\infty,\quad\forall n\in \mathbb{N}.
\end{equation}
We will show that 
\[
\lim_{\nu\to+\infty}\alpha_{n,\nu}=0 , \qquad n\in\mathbb{Z},
\]
by examining the individual terms appearing in $\alpha_{n,\nu}$. First, note that the hypothesis $\limsup_{\nu\to+\infty}|b_\nu/a_{\nu+1}|<1$ yields some numbers $r<1$ and $N\geq 1$ such that  
\begin{equation}\label{dd}
\big|b_\nu/a_{\nu+1}\big|<r, \qquad \nu\geq N.
\end{equation}

\noindent \textbf{Case $n\le -\nu$:} Since $A_{\nu+n,\nu}
=w_{\nu+n}\cdots w_{n+1}\frac{a_{\nu+n}}{a_{n}}$, the assumptions \eqref{yy} give that the family $\{A_{\nu+n,\nu}:\nu\ge 1\}$ is bounded for each fixed $n\le -\nu$.

\noindent{\bf Case $-\nu+1\le n\le 0$:}
We have
$
C_{\nu+n-1,\nu}
   = w_{\nu+n}\cdots w_{n+1} \frac{b_{\nu+n-1}}{a_{n}},
~\text{and}~
A_{\nu+n,\nu}=
   w_{\nu+n}\cdots w_{n+1}\frac{a_{\nu+n}}{a_{n}}.
$
Using the equations \eqref{yy}, \eqref{ddd}, and the estimate 
$|b_\nu/a_{\nu+1}|<r$, both families
$
\{C_{\nu+n-1,\nu}:\nu\ge 1\}$ and 
$\{A_{\nu+n,\nu}:\nu\ge 1\}
$
become bounded for fixed $n$.

\noindent{\bf Case $n\ge 1$:}
 Note that $A_{\nu+n,\nu}=w_{\nu+n}\cdots w_{n+1}\frac{a_{\nu+n}}{a_{n}}$,
and so, $\{A_{\nu+n,\nu}:\nu\geq 1\}$ is bounded for fixed $n$. Finally, we have
\[
\begin{aligned}
|C_{\nu+n-1,\nu}|
&\le
\left|\frac{w_{\nu+n}\cdots w_{n+1}a_{\nu+n}}{a_{n}}\right|
   \left|\frac{b_{\nu+n-1}}{a_{\nu+n}}\right|+
\left|\frac{w_{\nu+n-1}\cdots w_{n}a_{\nu+n-1}}{a_{n-1}}\right|
   \left|\frac{b_{n-1}}{a_{n}}\right|.
\end{aligned}
\]
Both the terms are bounded, and therefore  
$
\{C_{\nu+n-1,\nu}:\nu\geq 1\}
$
is bounded. Using \eqref{dd} and H\"{o}lder's inequality, in all three cases we find that, for large $\nu$
$$
|\alpha_{n,\nu}|\leq\begin{cases}
   C_{1}|\lambda_{\nu+n}|,\hspace{3.4cm} \text{if } n\leq -\nu,\\ 
   C_{2}(|\lambda_{\nu+n-1}|^{  p}+|\lambda_{\nu+n}|^{p})^{\frac{1}{p}},\hspace{.6cm} \text{if }-\nu+1\leq n\leq 0,\\ 
    C_{3} \left(\sum_{j\geq \nu-1}^{\nu+n}|\lambda_j|^p\right)^{1/p}, \hspace{1.2cm} \text{if }n\geq 1,\\ 
    \end{cases}
    $$
where $C_{1},C_{2},C_{3}$ are constants depending only on $r$. Combining all these inequalities, we conclude that the coefficients in \eqref{ccc} converge to $0$, as $\nu\rightarrow +\infty$.
\end{proof} 

The above dichotomy result on the orbits of $B_w$ yields a zero-one law for hypercyclicity:

\begin{corollary}
Under the hypothesis of the previous theorem, we have the following equivalent statements, $1\leq p<\infty$. 
\begin{itemize}
\item[(i)] $B_w$ is hypercyclic on $\ell^p_{a,b}(\Omega_{r,R}).$
\item[(ii)] $B_w$ has some orbit with a non-zero weak sequential limit point.
\item[(iii)] $B_w$ has an orbit admitting a non-zero norm limit point.
\end{itemize}
\end{corollary}

In view of the failure of the zero-one law for hypercyclicity, (cf. Proposition \ref{Not dense}), we have the following.

\begin{corollary}
There exists a weighted shift $B_w$ on a space of the form $\ell^p_{a,b}(\Omega_{r,R})$ such that $B_w$ is not similar to any classical weighted shift on $\ell^p(\mathbb{Z})$.
\end{corollary}

\begin{proof}
    Since, the zero-one law holds for classical weighted shifts, the result is immediate.
\end{proof}

\begin{corollary}
    There exists a weighted shift $B_w$ and a compact operator $K$ on $\ell^p(\mathbb{Z})$ such that $B_w+K$ is not hypercyclic, but it has orbits admitting non-zero limit points. 
\end{corollary}

{\bf Funding:} Bibhash Kumar Das is financially supported by the CSIR research fellowship (File No.: 09/1059(0037)/2020-EMR-I). Aneesh Mundayadan is partially funded by a Start-Up Research Grant of SERB-DST (File. No.: SRG/2021/002418).

\bibliographystyle{amsplain}

\begin{thebibliography}{99}
\bibitem{Abakumov-Abbar} E. Abakumov and A. Abbar, \textit{Orbits of the backward shifts with limit points}, J. Math. Anal. Appl. 537 (2024), no. 2, Paper No. 128293, 21 pp.
\bibitem{Adams-McGuire} G.T. Adams and P.J. McGuire, \textit{Analytic tridiagonal reproducing kernels}, J. Lond. Math. Soc. (2) 64 (2001), no. 3, 722–738.
%\bibitem{Adams-McGuire1} G.T. Adams and P.J. McGuire, \textit{A class of tridiagonal reproducing kernels}, Oper. Matrices 2 (2008) 233–247.
%\bibitem{Adams-MacGuire-Paulsen} G.T. Adams, P.J. McGuire and V. Paulsen, \textit{Analytic reproducing kernels and multiplication operators}, Ill. J. Math. 36 (1992) 404–419.
%\bibitem{Aro} N. Aronszajn, \textit{Theory of reproducing kernels}, Trans. Amer. Math. Soc. 68 (1950) 337–404.
\bibitem{Bayart-Grivaux} F. Bayart and S. Grivaux, \textit{Frequently hypercyclic operators}, Trans. Amer. Math. Soc. 358 (2006), no. 11, 5083–5117.

\bibitem{Bayart-Matheron}
F.\ Bayart and \'{E}.\ Matheron, \textit{Dynamics of Linear Operators},
Camb. Univ. Press. 179 (2009).

%\bibitem{Bayart-Ruzsa} F. Bayart and I.Z. Ruzsa, \textit{Difference sets and frequently hypercyclic weighted shifts}, Ergod. Theor. Dyn. Syst. 35 (2015), no. 3, 691–709.

%\bibitem{Bes-Peris} J. B\'{e}s and A. Peris, \textit{Hereditarily hypercyclic operators}, J. Funct. Anal. 167 (1999) 94–112.

%\bibitem{Beise-Muller1} H.-P. Beise and J. M\"{u}ller, \textit{Generic boundary behaviour of Taylor series in Hardy and Bergman spaces}, Math. Z. 284 (2016) 1185–1197.

%\bibitem{Beise-Meyrath-Muller} H.-P. Beise, T. Meyrath and J. M\"{u}ller, \textit{Mixing Taylor shifts and universal Taylor series}, Bull. Lond. Math. Soc. 47 (2015) 136–142.
\bibitem{Bonet1} J. Bonet, \textit{Dynamics of the differentiation operator on weighted spaces of entire functions}, Math. Z. 261 (2009), no. 3, 649–657.
\bibitem{Bonet2} J. Bonet, T. Kalmes, and A. Peris, \textit{Dynamics of shift operators on non-metrizable sequence spaces}, Rev. Mat. Iberoam. 37 (2021), no. 6, 2373-2397.
\bibitem{Bonilla-Cardeccia} A.\ Bonilla, R. Cardeccia, K.-G.\ Grosse-Erdmann, and S. Muro, \textit{Zero-one law of orbital limit points for weighted shifts}, Proc. Edinb. Math. Soc. (2) 68 (2025), no. 3, 945–978.
\bibitem{Bonilla-Erdmann1} A.\ Bonilla and K.-G.\ Grosse-Erdmann,
\textit{Frequently hypercyclic operators and vectors}, Ergod. Theor. Dyn. Syst. 27 (2007), no. 2, 383–404.
%\bibitem{Bourdon-Shapiro} P.S. Bourdon and J.H.Shapiro, \textit{Hypercyclic operators that commute with the Bergman backward shift}, Trans. Amer. Math. Soc. 352 (2000), no. 11, 5293 - 5316.
\bibitem{Kit} K. Chan and I. Seceleanu, \textit{Orbital limit points and hypercyclicity of operators on analytic function spaces}, Math. Proc. R. Ir. Acad. 110A (2010), no. 1, 99–109.
\bibitem{Kit1} K. Chan and I. Seceleanu, \textit{Hypercyclicity of shifts as a zero-one law of orbital limit points}, J. Operator Theory 67 (2012), no. 1, 257–277.
\bibitem{Costakis} G. Costakis and M. Sambarino, \textit{Topologically mixing hypercyclic operators}, Proc. Amer. Math. Soc. 132 (2004), no. 2, 385–389. 
%\bibitem{Curto} R. Curto and N. Salinas, \textit{Generalized Bergman kernels and the Cowen-Douglas theory}, Amer. J. Math. 106 (1984) 447–488.
\bibitem{Das-Mundayadan} B. K. Das and A. Mundayadan, \textit{Dynamics of weighted backward shifts on certain analytic function spaces}, Results Math. 79 (2024) no.7, paper No. 242, 29 pp.
\bibitem{Das-Mundayadan1} B. K. Das and A. Mundayadan,  \textit{Linear dynamics of the adjoint of a unilateral weighted shift operator,} https://doi.org/10.48550/arXiv.2412.05509.
\bibitem{Das-Sarkar1} S. Das and J. Sarkar, \textit{Tridiagonal shifts as compact $+$ isometry}, Arch. der. Math. 119 (2022), no. 5, 507–518. 
%\bibitem{Das-Sarkar} S. Das and J. Sarkar, \textit{Tridiagonal kernels and left-invertible operators with applications to Aluthge transforms}, Rev. Mat. Iberoam. 39(2023), no. 2, 397–437.
\bibitem{Diestel} J. Diestel, \textit{Sequences and Series in Banach Spaces}, Grad. Texts in Math., Springer-Verlag, 1984.
\bibitem{Douglas} R.G. Douglas, \textit{Banach Algebra Techniques in Operator Theory}, New York: Academic, 1972.

%\bibitem{Duren} P. Duren, \emph{Theory of $H^p$ spaces}, Pure and Applied Mathematics 38, Academic Press, New York, 1970.
\bibitem{Gethner-Shapiro} R.M. Gethner and J.H. Shapiro, \textit{Universal vectors for operators on spaces of holomorphic functions}, Proc. Amer. Math. Soc. 100 (1987), no. 2, 281–288. 
\bibitem{Godefroy-Shapiro} G. Godefroy and J.H. Shapiro, \textit{Operators with dense,
invariant, cyclic vector manifolds}, J. Funct. Anal. 98 (1991), no. 2, 229–269.
\bibitem{Gonzalez} M. Gonz\'{a}lez, F. Le\'{o}n-Saavedra, and A. Montes-Rodr\'{i}guez, \textit{Semi-Fredholm theory: hypercyclic and supercyclic subspaces}, Proc. Lond. Math. Soc. (3) 81 (2000), no. 1, 169–189. 

\bibitem{Grivaux} S. Grivaux, \'{E}. Matheron, and Q. Menet, \textit {Linear dynamical systems on Hilbert spaces: typical properties and explicit examples}, Mem. Amer. Math. Soc. 269 (2021), no.1315, v+147 pp.
\bibitem{Erdmann} K.-G. Grosse-Erdmann, \emph{Hypercyclic and chaotic weighted shifts}, Studia Math. 139 (2000), no. 1, 47–68.

\bibitem{Erdmann-Peris} K.-G. Grosse-Erdmann and A. Peris, \emph{Linear Chaos}, Springer Universitext,
2011.
\bibitem{Halmos} P.R. Halmos, \textit{A Hilbert Space Problem Book}, Second Edn., Grad. Texts in Math., vol. 19, Springer-Verlag, New-York, Berlin, 1982. 
\bibitem{Zhu} H. Hedenmalm, B. Korenblum, and K. Zhu, \emph{Theory of Bergman Spaces}, Springer-Verlag, New York, 2000.
\bibitem{Hilden-Wallen} H.M. Hilden and L.J. Wallen, \textit{Some cyclic and non-cyclic vectors of certain operators}, Indiana Univ. Math. J. 23 (1974), no. 7, 557-565.
\bibitem{Kitai} C. Kitai, \textit{Invariant closed sets for linear operators}, Ph.D thesis, University of Toronto, Toronto, 1982.
%\bibitem{Gimenez-Peris} F. Mart\'{i}nez-Gim\'{e}nez and A. Peris, \emph{Universality and chaos for tensor products
%of operators}, J. Approx. Theory, 124 (2003), 7-24.
%\bibitem{Zhu} H. Hedenmalm, B. Korenblum and K. Zhu, \emph{Theory of Bergman spaces}, Springer-Verlag, New York, 2000.


\bibitem{Leon} F. Le\'{o}n-Saavedra and A. Montes-Rodr\'{i}guez, \textit{Linear structure of hypercyclic vectors}, J. Funct. Anal. 148 (1997), 524–545.
\bibitem{Lindenstrauss} J. Lindenstrauss and L. Tzafriri, \textit{Classical Banach Spaces I}, Springer Verlag, 1970.
\bibitem{Menet} Q. Menet, \textit{Hypercyclic subspaces and weighted shifts}, Adv. Math. 255 (2014), 305-337.
\bibitem{Montes} A. Montes-Rodr\'{i}guez, \textit{Banach spaces of hypercyclic vectors}, Michigan Math. J. 43 (1996), no. 3, 419-436.
\bibitem{Montes1} A. Montes-Rodríguez and H.N. Salas, \textit{Supercyclic subspaces: spectral theory and weighted shifts}, Adv. Math. 163 (2001), no.1, 74–134.
%\bibitem{Muller-Maike} J. M\"{u}ller and T. Maike: \textit{Dynamics of the Taylor shift on Bergman spaces}, J. Oper. Theor. 87 (2022) 25–40.
\bibitem{Mundayadan-Sarkar} A. Mundayadan and J. Sarkar, \textit{Linear dynamics in reproducing kernel Hilbert spaces}, Bull. Sci. Math. 159 (2020), 102826, 29 pp.

%\bibitem{Paulsen} V. Paulsen and M. Raghupathi, \emph{An introduction to the theory of reproducing kernel Hilbert spaces}, Camb. Univ. Press 152, 2016.


%\bibitem{Nikolski} Nikolai K. Nikolski, \emph{Operators, functions, and systems: an easy reading. Vol. 1}, American Mathematical Society, Providence, RI, 2002. xiv+461 pp.

%\bibitem{Range} R.M. Range. \newblock {\em Holomorphic Functions and Integral Representations in Several	Complex Variables}, volume 108 of { Graduate Texts in Mathematics}.\newblock Springer-Verlag, New York, 1986.
\bibitem{Rolewicz} S. Rolewicz, \textit{On orbits of elements}, Studia Math. 32 (1969), 17–22.
\bibitem{Salas-hc} H.N. Salas, \textit{Hypercyclic weighted shifts}, Trans. Amer. Math. Soc. 347 (1995), no. 3, 993-1004.
\bibitem{Salas-sc} H.N. Salas, \textit{Supercyclicity and weighted shifts}, Studia Math. 135 (1999), no. 1, 55–74.
\bibitem{Shields} A.L. Shields, \textit{Weighted shift operators and analytic function theory, in: Topics of Operator Theory}, in: Math. Surveys Monogr., vol.13, American Math. Soc., Providence, RI, 1974, pp.49–128.
\end{thebibliography}

\end{document}